\documentclass[12pt]{article}
\usepackage{a4wide, amstext, amsmath, amssymb, graphicx, array, epsfig, psfrag}
\usepackage{amstext, amsmath, amssymb, graphicx, todonotes}
\usepackage{stmaryrd}
\usepackage[update,prepend]{epstopdf}

\newcommand{\bfI}{\boldsymbol I}

\newcommand{\bfbeta}{\boldsymbol \beta}

\newcommand{\mcK}{\mathcal{K}}
\newcommand{\mcE}{\mathcal{E}}
\newcommand{\mcN}{\mathcal{N}}
\newcommand{\mcF}{\mathcal{F}}
\newcommand{\mcA}{\mathcal{A}}
\newcommand{\mcT}{\mathcal{T}}

\newcommand{\mcH}{\mathcal{H}}

\newcommand{\lp}{\left(}
\newcommand{\rp}{\right)}
\newcommand{\tn}{|\mspace{-1mu}|\mspace{-1mu}|}

\newcommand{\Gammah}{{\Gamma_h}}
\newcommand{\nablas}{\nabla_\Gamma}
\newcommand{\nablash}{\nabla_{\Gamma_h}}

\newcommand{\divs}{\text{div}_\Gamma}
\newcommand{\divsh}{\text{div}_\Gammah}
\newcommand{\IR}{\mathbb{R}}

\newcommand{\Ps}{{P}_\Gamma}
\newcommand{\Psh}{{P}_\Gammah}

\numberwithin{equation}{section}
\newtheorem{lem}{Lemma}[section]
\newtheorem{prop}{Proposition}[section]
\newtheorem{thm}{Theorem}[section]
\newtheorem{rem}{Remark}[section]

\newenvironment{proof}{\noindent \newline {\bf Proof.}}
{\hfill \mbox{\fbox{} } \newline}
\begin{document}
\title{\bf Stabilized CutFEM for the Convection 
Problem on Surfaces
\thanks{This research was supported in part by the Swedish Foundation for Strategic Research Grant No.\ AM13-0029 (PH,MGL), the Swedish Research Council Grants Nos.\ 2011-4992 (PH), 2013-4708 (MGL), 
and 2014-4804 (SZ),
the Swedish Research Programme Essence (MGL, SZ), and EPSRC, UK, 
Grant Nr. EP/J002313/1. (EB)}
}
\author{Erik Burman
\footnote{Department of Mathematics, University of
Sussex, Brighton, BN1 9RF United Kingdom} 
%\footnote{Supported by EPSRC, UK, Grant Nr. EP/J002313/1}
\mbox{ }
 Peter Hansbo
\footnote{Department of Mechanical Engineering, J\"onk\"oping University,
SE-55111 J\"onk\"oping, Sweden.} 
%\footnote{Supported by the Swedish Foundation for Strategic Research Grant %Nr.\ AM13-0029 and the Swedish Research Council Grant Nr.\ 2011-4992.}
\mbox{ }
Mats G.\ Larson
\footnote{Department of Mathematics and Mathematical Statistics, Ume{\aa} University, SE-90187 Ume{\aa}, Sweden} 
%\footnote{Supported by the Swedish Foundation for Strategic Research Grant %Nr.\ AM13-0029 and the Swedish Research Council Grant Nr.\ 2013-4708.}
\mbox{ }
Sara Zahedi
\footnote{Department of Mathematics, KTH Royal Institute of Technology,
SE-100 44 Stockholm, Sweden}
}
\maketitle
\begin{abstract} We develop a stabilized cut finite element 
method for the convection problem on a surface based on continuous 
piecewise linear approximation and gradient jump stabilization 
terms. The discrete piecewise linear surface cuts through a background mesh consisting of tetrahedra in an arbitrary way 
and the finite element space consists of piecewise linear continuous functions defined on the 
background mesh. The variational form involves integrals on the surface 
and the gradient jump stabilization term is defined on the full faces of the tetrahedra. The stabilization term serves two purposes: first 
the method is stabilized and secondly the resulting linear system 
of equations is algebraically stable. We establish stability results that are analogous to the standard meshed flat case and 
prove $h^{3/2}$ order convergence in the natural norm 
associated with the method and that the full gradient enjoys $h^{3/4}$ order of convergence in $L^2$. We also show that the condition number of the stiffness matrix is bounded by $h^{-2}$. Finally, our results are verified by numerical examples.
\end{abstract}

\section{Introduction} In this contribution we develop a 
stabilized cut finite element for stationary convection on 
a surface embedded in $\IR^3$. The method is based on 
a three dimensional background mesh consisting of 
tetrahedra and a piecewise linear approximation of the surface. 
The finite element space is the continuous piecewise linear 
functions on the background mesh and the bilinear form defining the 
method only involves integrals on the surface. In addition we 
add a consistent stabilization term which involves the 
normal gradient jump on the full faces of the background 
mesh. In the case of the Laplace-Beltrami operator the idea 
of using the restriction of a finite element space to the 
surface was developed in \cite{OlReGr09}, and a stabilized 
version was proposed and analyzed in \cite{BuHaLa14}.

Surprisingly, we show here that for the convection problem 
the properties of cut finite element method completely 
reflects the properties of the corresponding method on 
standard triangles or tetrahedra, see the analysis for the 
latter in \cite{BuHa04}. In particular, we prove discrete 
stability estimates in the natural 
energy norm, involving the $L^2$ norm of the solution and 
$h^{1/2}$ times the $L^2$ norm of the streamline derivative 
where $h$ is the meshsize, and corresponding optimal a priori 
error estimates of order $h^{3/2}$. Furthermore, we also 
show an error estimate of order $h^{3/4}$ for the error 
in the full gradient, which is better than the streamline 
diffusion method on triangles, and is also in line with 
\cite{BuHa04}. The stabilization term is key 
to the proof of the discrete stability estimates and enables us 
to work in the natural norms corresponding to those used in 
the standard analysis on triangles or tetrahedra. The analysis 
utilizes a covering argument first developed in \cite{BuHaLa14}, 
which essentially localizes the analysis to sets of elements, 
with a uniformly bounded number of elements, that together 
has properties similar to standard finite elements. The stabilization 
term also leads to an algebraically stable linear system of 
equations and we prove that the condition number is bounded by $h^{-2}$.

We note that similar stabilization terms have recently been 
used for stabilization of cut finite element methods for time 
dependent problems in \cite{HaLaZa15}, bulk domain 
problems involving standard boundary and interface 
conditions \cite{BuHa10}, \cite{BuHa12}, \cite{HaLaZa14}, 
and \cite{MaLaLoRo12}, and for coupled bulk-surface problems 
involving the Laplace-Beltrami operator on the surface in \cite{BuHaLaZa14}. We also mention \cite{BuHaLaMa15} 
where a discontinuous cut finite element method for the 
Laplace-Beltrami operator was developed. None of these 
references consider the convection problem on the surface. 
For convection problems streamline diffusion stabilization 
was used in \cite{OlReXu14b} and methods on evolving 
surfaces were studied in \cite{OlRe14} and \cite{OlReXu14a}.

Finally, we refer to \cite{DeDzElHe10}, \cite{Dz88}, 
\cite{DzEl08}, and \cite{DzEl13} for general background on 
finite element methods for partial differential equations 
on surfaces.

The outline of the reminder of this paper is as follows: In 
Section 2 we formulate the model problem; in Section 3 we 
define the discrete surface, and its approximation properties, 
and the finite element method; in Section 4 we summarize 
some preliminary results involving lifting of functions 
from the discrete surface to the continuous surface;
in Section 5 we first derive some technical lemmas essentially quantifying the stability induced by the stabilization 
term, and then we derive the key discrete stability estimate; 
in Section 6 we prove a priori estimates; in Section 7 
we prove an estimate of the condition number; and finally 
in Section 8, we present some numerical examples 
illustrating the theoretical results.

\section{The Convection Problem on a Surface}

\subsection{The Surface} Let $\Gamma$ be a surface embedded in $\IR^3$ 
with signed distance function $\rho$ such that the exterior unit 
normal to the surface is given by $n = \nabla \rho$. We let 
$p:\IR^3 \rightarrow \Gamma$ be the closest point mapping. Then there 
is a $\delta_0>$ such that $p$ maps each point in 
$U_{\delta_0}(\Gamma)$ to precisely one point on $\Gamma$, where 
$U_\delta(\Gamma) = \{ x \in \IR^3 : | \rho(x)|  < \delta\}$ is the 
open tubular neighborhood of $\Gamma$ of thickness $\delta$.

\subsection{Tangential Calculus} For each function $u$ on 
$\Gamma$ we let the extension $u^e$ to the neighborhood 
$U_{\delta_0}(\Gamma)$ be defined by the pull back $u^e = u \circ p$. 
For a function $u: \Gamma \rightarrow \IR$ we then define the 
tangential gradient
\begin{equation}
\nablas u = \Ps \nabla u^e
\end{equation}
where $\Ps = I - n \otimes n$, with $n=n(x)$, is the projection 
onto the tangent plane $T_{x}(\Gamma)$. We also define the surface 
divergence 
\begin{equation}
\divs(u) = \text{tr}(u \otimes \nablas ) 
= \text{div}(u^e) - n \cdot (u^e \otimes \nabla)\cdot n 
\end{equation}
where $(u\otimes \nabla)_{ij} = \partial_j u_i$. It can be shown that the tangential derivative does not depend on the particular choice of extension.

\subsection{The Convection Problem on $\boldsymbol{\Gamma}$}
The strong form of the convection problem on $\Gamma$ takes 
the form: find $u:\Gamma \rightarrow \IR$ such that
\begin{alignat}{2}\label{eq:conva}
\beta \cdot \nablas u + \alpha u &= f \qquad &\text{on $\Gamma$}
\end{alignat}  
where $\beta:\Gamma \rightarrow \IR^3$ is a given tangential vector 
field, $\alpha:\Gamma \rightarrow \IR$ and $f:\Gamma \rightarrow \IR$ are given functions. 

\paragraph{Assumption.} The coefficients 
$\alpha$ and $\beta$ are smooth and satisfy 
\begin{equation}\label{eq:assumdivbeta}
0< C \leq \inf_{x\in \Gamma} (\alpha(x) - \frac{1}{2} \divs \beta(x) ) 
\end{equation}
for a positive constant $C$. 

We introduce the Hilbert space $V = \{ v:\Gamma \rightarrow \IR : 
\| v \|^2_V = \|v\|_\Gamma^2 + \|\beta\cdot \nabla v\|_\Gamma < \infty\}$ and the operator $L v = \beta \cdot \nabla v + \alpha v$. We note that using Green's formula and assumption (\ref{eq:assumdivbeta}) we have the estimate
\begin{equation}
(Lv,v)_\Gamma = ((\alpha - \frac{1}{2}\divs \beta)v,v)_\Gamma 
\geq C \|v\|_\Gamma^2 
\end{equation}

\begin{prop}\label{prop:existence} If the coefficients $\alpha$ and $\beta$ satisfy assumption (\ref{eq:assumdivbeta}), then there is a unique 
$u \in V$ such that $L u = f$ for each $f\in L^2(\Gamma)$.   
\end{prop}
\begin{proof} The essential idea in the proof is to consider the corresponding time dependent problem with a 
smooth right hand side and show that the solution exists 
and converges to a solution to the stationary problem as 
time tends to infinity. Then we use a density argument to 
handle a right hand side in $L^2$. 
\paragraph{Smooth Right Hand Side.}
For any $0< T <\infty$ consider the time 
dependent problem: find $u:[0,T]\times \Gamma \rightarrow \IR$, 
such that 
\begin{equation}\label{eq:timedep}
u_t + \beta \cdot \nabla u + \alpha u = g 
\quad \text{on $(0,T] \times \Gamma$},
\qquad u(0) = 0 \quad \text{on $\Gamma$}
\end{equation}
Consider first a smooth right hand side $g$. Using characteristic coordinates we conclude that there is smooth solution 
$u(t)$ to (\ref{eq:timedep}). Next taking the time derivative 
of equation (\ref{eq:timedep}) we find that the solution satisfies the equation
\begin{equation}\label{eq:timedepder}
u_{tt} + \beta \cdot \nabla u_t + \alpha u_t = 0 
\end{equation}
where we used the fact that $\alpha$, $\beta$, and $g$, 
do not depend on time. Multiplying (\ref{eq:timedepder}) 
by $u_t$ and integrating over $\Gamma$ we get 
\begin{equation}
\frac{d}{dt} \| u_t \|^2_\Gamma 
+ 2 ((\alpha - \divs \beta ) u_t,u_t) _\Gamma
= 0
\end{equation}
Using (\ref{eq:assumdivbeta}) we obtain
\begin{equation}
\frac{d}{dt} \| u_t \|^2_\Gamma + 2 C \|u_t \|^2_\Gamma \leq 0
\end{equation}
which implies 
\begin{equation}
\frac{d}{dt} \Big( \| u_t \|^2_\Gamma e^{2C t} \Big) \leq 0
\end{equation}
Integrating over $[0,T]$ we get 
\begin{equation} \label{eq:exista}
\| u_t(T) \|_\Gamma \leq \| u_t(0) \|_\Gamma e^{-C T} 
= \| g \|_\Gamma e^{-C T}
\end{equation}
where we used equation (\ref{eq:timedep}) for $t=0$ 
to conclude that $u_t(0)=g$ since $u(0)=0$ and therefore 
also $\nablas u(0) = 0$. Using (\ref{eq:exista}) we have 
\begin{align}
&\| u(T_2) - u(T_1) \|_\Gamma 
= \|\int_{T_1}^{T_2} u_t(s)ds \|_\Gamma 
\leq  \int_{T_1}^{T_2} \| u_t(s) \|_\Gamma ds 
\\ \label{eq:existb}
&\qquad \leq \int_{T_1}^{T_2} \| g \|_\Gamma e^{-C s} ds
\leq e^{-CT_1}\Big(1 + e^{-C(T_2-T1)}\Big) \| g \|_\Gamma 
\leq e^{-CT_1} \| g \|_\Gamma
\end{align}
for $0\leq T_ 1\leq T_2< \infty$. Using the time dependent 
equation (\ref{eq:timedep}) we have
\begin{equation}
\beta\cdot \nabla (u(T_2) - u(T_1)) = u_t(T_1) - u_t(T_2) 
+ \alpha (u(T_1) - u(T_2))
\end{equation}
and therefore, using the fact that 
$\|\alpha\|_{L^\infty(\Gamma)} \lesssim 1$, we have the estimate
\begin{align}
\|\beta\cdot \nabla (u(T_2) - u(T_1))\|_\Gamma 
&\lesssim \|u_t(T_1)\|_\Gamma + \|u_t(T_2)\|_\Gamma
+\|(u(T_2) - u(T_1))\|_\Gamma 
\\  \label{eq:existc}
&\lesssim e^{-C T_1} \|g \|_\Gamma
\end{align}
where we used (\ref{eq:exista}) and (\ref{eq:existb}) in the 
last step. Together, (\ref{eq:existb}) and (\ref{eq:existc}) 
leads to the estimate
\begin{equation}\label{eq:existd}
\| u(T_2) - u(T_1) \|_V \lesssim e^{-C T_1} \|g \|_\Gamma, \qquad
0\leq T_1 \leq T_2
\end{equation}
Thus we conclude that for each $\epsilon > 0 $ there is 
$T_\epsilon$ such that $\| u(T_1) - u(T_2) \|_V \leq \epsilon$ 
for all $T_1,T_2> T_\epsilon$. We can then pick a sequence 
$u(T_n)$ with $T_n=n, n=1,2,3,\dots$ and conclude 
from (\ref{eq:existd}) that the sequence is Cauchy in 
$V$ and therefore it converges to a limit $u_{g} \in V$. 
We then have 
\begin{equation}
\|L u_g - g\|_\Gamma 
\leq \|L u_g - L u_n\|_\Gamma + \|L u_n - g\|_\Gamma 
\leq \|u_g - u_n\|_V + \|u_t(t_n)\|_\Gamma
\leq e^{-C T_1} \|g \|_\Gamma
\end{equation}
and thus the limit $u$ is a solution to the stationary problem in 
the sense of $L^2$ and from (\ref{eq:existd}) with $T_1=0$, we have the stability estimate 
\begin{equation}\label{eq:existe}
\| u_g \|_V \lesssim \| g \|_\Gamma
\end{equation}

\paragraph{Right Hand Side in $\boldsymbol{L^2(\Gamma)}$.} For 
$f \in L^2(\Gamma)$ we pick a sequence of smooth functions $f_n$ 
that converges to $f$ in $L^2(\Gamma)$. Then for each $f_n$ 
there is a solution $u_n \in V$ to $L u_n = f_n$ 
and we note that $L(u_n - u_m) = f_n - f_m$ and therefore it 
follows from (\ref{eq:existe}) that 
\begin{align}
\| u_n - u_m \|_V \lesssim \|f_n - f_m\|_\Gamma
\end{align}
and thus $\{u_n\}$ is a Cauchy sequence since $\{f_n\}$ 
is a Cauchy sequence. Denoting the limit of $u_n$ by $u$ we 
have
\begin{equation}
\|L u - f \|_\Gamma 
\leq \| L (u - u_n) \|_\Gamma + \|f_n- f\|_\Gamma
\leq \| u - u_n \|_V + \| f_n - f\|_\Gamma
\end{equation}
which tends to zero as $n$ tends to infinity and thus $u \in V$ 
is a solution to $L u = f$ in the sense of $L^2$.
\end{proof}

\section{The Finite Element Method}

\subsection{The Discrete Surface}
Let $\Omega_0$ be a polygonal domain that contains $U_{\delta_0}(\Gamma)$ 
and let $\{\mcT_{0,h}, h \in (0,h_0]\}$ be a family of quasiuniform partitions 
of $\Omega_0$ into shape regular tetrahedra with mesh parameter $h$. 
Let $\Gamma_h\subset \Omega_0$ be a connected surface such that 
$\Gamma_h \cap T$ is a subset of some hyperplane for each 
$T\in \mcT_{0,h}$ and let $n_h$ be the piecewise constant unit normal 
to $\Gamma_h$. 

\paragraph{Geometric Approximation Property.} The family $\{\Gamma_h: h \in (0,h_0]\}$ 
approximates $\Gamma$ in the following sense: 
\begin{itemize}
\item $\Gamma_h \subset U_{\delta_0}(\Gamma)$, $\forall h \in (0,h_0]$, and the closest point mapping $p:\Gamma_h \rightarrow \Gamma$ is a 
bijection.
\item The following estimates hold
\begin{equation}\label{assum:geom}
\|\rho\|_{L^\infty(\Gamma_h)} \lesssim h^2, \qquad 
\|n - n_h \|_{L^\infty(\Gamma_h)} \lesssim h
\end{equation}
\end{itemize}

We introduce the following notation for the geometric entities involved in the mesh
\begin{align}
\mcT_h &= \{T\in \mcT_{h,0} : \overline{T}\cap \Gammah \neq \emptyset\}
\\
\mcF_h &= \{F = (\overline{T}_1 \cap \overline{T}_2)
\setminus 
\partial (\overline{T}_1 \cap \overline{T}_2) : T_1,T_2 \in \mcT_h\}
\\
\mcK_h &=\{K = T \cap \Gammah : T \in \mcT_h\}
\cup \{F \in \mcF_h: F \subset \Gamma_h\}
\\
\mcE_h &= \{E= \partial K_1 \cap \partial K_2 : K_1,K_2 \in \mcK_h\} 
\end{align}
We also use the notation $\omega^l = \{p(x)\in \Gamma : x \in 
\omega \subset \Gamma_h \}$, in particular, $\mcK_h^l = 
\{K^l : K\in \mcK_h^l\}$ is a partition of $\Gamma$.

\subsection{The Finite Element Method}
We let $V_h$ be the space of continuous piecewise linear 
functions defined on $\mcT_h$. The finite element method 
takes the form: find $u_h \in V_h$ such that 
\begin{equation}\label{eq:fem}
A_h(u_h,v) = l_h(v) \quad \forall v \in V_h
\end{equation}
Here the forms are defined by
\begin{equation}
A_h(v,w) = a_h(v,w) + j_h(v,w), \qquad l_h(v) = (f,v)_\Gammah
\end{equation}
and
\begin{align}
a_h(v,w) &= (\beta_h \cdot\nablash v,w)_\Gammah 
+ (\alpha_h v,w)_\Gammah 
\\
j_h(v,w) &= c_F h ([n_F \cdot \nablash v],[n_F \cdot \nablash w])_{\mcF_h}
\end{align}
where $\nablash v = \Psh \nabla v = (I - n_h\otimes n_h) \nabla v$ is the elementwise defined tangent 
gradient on $\Gammah$, $c_F$ is a positive stabilization parameter, $\alpha_h$ and $\beta_h$ are discrete approximations of $\alpha$ and $\beta$ which we specify in Section \ref{Sec:QuadErrors} below. 

\section{Preliminary Results}
\subsection{Norms}
We let $\| v \|_\omega$ denote the $L^2$ norm over the set 
$\omega$ equipped with the appropriate Lebesgue measure. 
Furthermore, we introduce the scalar products
\begin{align}
(v,w)_{\mcT_h} &= \sum_{T\in\mcT_h} (v,w)_T
,\quad
(v,w)_{\mcK_h} = \sum_{K\in\mcK_h} (v,w)_K
\\
(v,w)_{\mcF_h} &= \sum_{F\in\mcF_h} (v,w)_F
,\quad
(v,w)_{\mcE_h} = \sum_{E\in\mcE_h} (v,w)_E
\end{align}
with corresponding $L^2$ norms denoted by $\|\cdot\|_{\mcT_h}, \|\cdot\|_{\mcK_h}, \|\cdot\|_{\mcF_h},$ and $\|\cdot\|_{\mcE_h}$. 
Note that $\|\cdot \|_{\mcK_h} = \|\cdot \|_{\Gammah}$ and 
that the following scaling relations hold
\begin{equation}
\sum_{T\in \mcT_h} |T| \sim h,
\quad 
\sum_{K \in \mcK_h} |K|\sim \sum_{F \in \mcF_h} |F| \sim 1,
\quad
\sum_{E \in \mcE_h} |E| \sim h^{-1}
\end{equation}
Finally, we introduce the energy type norms
\begin{align}\label{eq:normtrippleh}
\tn v \tn_h^2 &= \tn v \tn^2_{\mcK_h} + h \tn v \tn^2_{\mcF_h}
\\ \label{eq:normtrippleKh}
\tn v \tn^2_{\mcK_h} &= h\|\beta_h \cdot \nablash v \|_{\mcK_h}^2  
+ \| v \|^2_{\mcK_h} 
\\ \label{eq:normtrippleFh}
\tn v \tn^2_{\mcF_h} &= \| [ n_F \cdot \nabla v ]\|^2_{\mcF_h}
\end{align}

\subsection{Inverse Estimates} 
Let $T \in \mcT_h$, $K=\Gamma_h \cap T$, 
$E \in \mcE_h$ and $E\subset \partial K$, then 
the following inverse estimates hold
\begin{align}\label{eq:inverseEtoF}
h \| v \|^2_{E} & \lesssim \|v \|^2_{F} \quad \forall v \in V(F)
\\ \label{eq:inverseFtoT}
h \|v \|^2_{F} &\lesssim \| v \|^2_{T}\quad \forall v \in W(T)
\\ \label{eq:inverseKtoT}
h \|v \|^2_{K} &\lesssim \| v \|^2_{T}\quad \forall v \in W(T)
\end{align}
with constants independent of the position of the 
intersection of $\Gammah$ and $T$. Note that the second 
inequality is the standard element to face inverse inequality.
Here $V(F) = \widehat{V}\circ X_F^{-1}$ ($W(T) = \widehat{W}\circ X_T^{-1}$), where $\widehat{V}$ ($\widehat{W}$) is a finite dimensional space on the reference triangle $\widehat{F}$ (reference tetrahedron $\widehat{T}$) and $X_F:\widehat{F}\rightarrow F$ ($X_T:\widehat{T}\rightarrow T$) an affine bijection.

\subsection{Extension and Lifting of Functions}
In this section we summarize basic results concerning extension and liftings of functions. We refer to \cite{BuHaLa14} and \cite{De09} for 
further details.

\paragraph{Extension.}
Recalling the definition of the extension and using the chain rule 
we obtain the identity 
\begin{equation}\label{eq:tanderext}
\nablash v^e = B^T \nablas v  
\end{equation}
where 
\begin{equation}\label{Bmap}
B = P_\Gamma (I - \rho \mcH)  P_\Gammah : T_{x}(K)\rightarrow
T_{p(x)} (\Gamma)
\end{equation}
and $\mcH = \nabla \otimes \nabla \rho$. Here $\mcH$ is a 
$\Gamma$-tangential tensor, which equals the 
curvature tensor on $\Gamma$, and there is $\delta>0$ such 
that $\|\mcH \|_{L^\infty(U_\delta(\Gamma))}\lesssim 1$. 
Furthermore, we have the inverse mapping $B^{-1}= P_\Gammah(I-\rho\mcH)^{-1} P_\Gamma : T_{p(x)}(\Gamma)\rightarrow T_{x} (K)$. 
%We refer to \cite{GiTr83}, Section XX for further details.

\paragraph{Lifting.}
The lifting $w^l$ of a function $w$ defined on $\Gamma_h$ to $\Gamma$ 
is defined as the push forward
\begin{equation}
(w^l)^e = w^l \circ p = w \quad \text{on $\Gamma_h$}
\end{equation}
and we have the identity 
\begin{equation}\label{eq:tanderlift}
\nablas w^l = B^{-T} \nablash w
\end{equation}

\paragraph{Estimates Related to $\boldsymbol{B}$.}
Using the uniform bound  
$\|\mcH\|_{L^\infty(U_\delta(\Gamma))}\lesssim 1$ 
it follows that
\begin{equation}\label{BBTbound}
  \| B \|_{L^\infty(\Gamma_h)} \lesssim 1,
 \quad \| B^{-1} \|_{L^\infty(\Gamma)} \lesssim 1,
 \quad \| (P_\Gamma - B)P_\Gammah \|_{L^\infty(\Gamma_h)} 
 \lesssim h^2
\end{equation}
Next consider the surface measure $d \Gamma = |B| d \Gammah$,
where $|B|$ is the absolute value of the determinant of $[B \xi_1 \, B \xi_2\, n^e]$ and $\{\xi_1,\xi_2\}$ is an orthonormal basis in 
$T_x(K)$. We have the following estimates
\begin{equation}\label{detBbound}
\| 1 - |B| \|_{L^\infty(\Gamma_h)} \lesssim h^2, \quad \||B|\|_{L^\infty(\Gamma_h)} \lesssim 1, \quad \||B|^{-1}\|_{L^\infty(\Gamma_h)} \lesssim 1
\end{equation}
In view of these bounds and the identities (\ref{eq:tanderext}) and 
(\ref{eq:tanderlift}) we obtain the following equivalences
\begin{equation}\label{eq:normequ}
\| v^l \|_{L^2(\Gamma)}
\sim \| v \|_{L^2(\Gammah)}, \qquad \| v \|_{L^2(\Gamma)} \sim
\| v^e \|_{L^2(\Gammah)}
\end{equation}
and
\begin{equation}\label{eq:normequgrad}
\| \nabla_\Gamma v^l \|_{L^2(\Gamma)} \sim \| \nablash v \|_{L^2(\Gammah)},
\qquad \| \nablas v \|_{L^2(\Gamma)} \sim
\| \nablash v^e \|_{L^2(\Gammah)}
\end{equation}

\subsection{Interpolation}
Let $\pi_h:L^2(\mcT_h) \rightarrow V_h$ be the Cl\'ement interpolant. Then
we have the following standard estimate
\begin{equation}\label{eq:interpoltets}
\| v - \pi_h v \|_{H^m(T)} \lesssim h^{s-m} \| v \|_{H^s(\mcN(T))},
\quad m\leq s \leq 2, \quad m=0,1
\end{equation}
where $\mcN(T)\subset \mcT_h$ is the set of neighboring elements of $T$. In particular, we have the $L^2$ stability estimate
\begin{equation}
\| \pi_h v \|_{T} \lesssim \| v \|_{\mcN(T)}\quad \forall T\in \mcT_h
\end{equation}
and as a consequence $\pi_h:L^2(\mcT_h) \rightarrow V_h$ 
is uniformly bounded and we have the estimate 
\begin{equation}\label{eq:pihl2stab}
\| \pi_h v \|_{\mcT_h} \lesssim \| v \|_{\mcT_h}
\end{equation}
Using the trace inequality
\begin{equation}
\| v \|^2_{T\cap \Gammah} \lesssim h^{-1} \| v \|_T^2 + h \| \nabla v \|^2_T
\end{equation}
where the constant is independent of the position of the intersection between $\Gammah$ and $T$, see \cite{HaHaLa04} 
for a proof, and stability of the extension operator 
\begin{equation}\label{eq:extstab}
\|v^e\|_{H^s(U_{\delta_0}(\Gamma))}
\lesssim \delta^{1/2}\|v\|_{H^s(\Gamma)}
\end{equation}
we obtain the interpolation error estimate
\begin{equation}\label{eq:interpolsurf}
\| v - \pi_h v \|_{H^m(\mcK_h)} \lesssim h^{s-m} \| v \|_{H^s(\Gamma)} \quad m\leq s \leq 2, \quad m=0,1
\end{equation}
Using (\ref{eq:interpolsurf}) and the definition of the energy 
norm (\ref{eq:normtrippleKh}) we obtain
\begin{equation}\label{eq:interpolenergysurf}
\tn v^e - \pi_h v^e \tn_{\mcK_h} \lesssim h^{3/2}\|v\|_{H^2(\Gamma)}
\end{equation}
and using a standard trace inequality on tetrahedra, the interpolation estimate (\ref{eq:interpoltets}), and the stability (\ref{eq:extstab}) of the extension, we have
\begin{equation}
 \label{eq:interpolface}
\tn v^e - \pi_h v^e \tn_{\mcF_h} \lesssim h \|v\|_{H^2(\Gamma)}
\end{equation}
Combining these two estimates we get
\begin{equation}\label{eq:interpolenergy}
\tn v^e - \pi_h v^e \tn_{h} \lesssim h^{3/2}\|v\|_{H^2(\Gamma)}
\end{equation}

\section{Stability Estimates}

\subsection{Some Technical Lemmas}
In this section we derive some technical lemmas that provide bounds for 
certain terms that arise in the stability estimates. The bounds depend 
in a critical way on the additional control provided by the gradient 
jump stabilization term. We begin by recalling a construction of coverings 
developed in \cite{BuHaLa14}, Section 4.1.

\paragraph{Families of Coverings of $\mcT_h$.} Let $x$ be a point on 
$\Gamma$ and let $B_\delta(x) = \{y\in \IR^3 : |y-x|<\delta\}$ and 
$D_\delta(x) = B_\delta(x)\cap \Gamma$. We define the sets of elements
\begin{equation}
\mcK_{\delta,x} = \{K \in \mcK_h : \overline{K}^l \cap D_\delta(x)\neq \emptyset \}, \qquad 
\mcT_{\delta,x} = \{T\in \mcT_h : T\cap \Gammah \in \mcK_{\delta,x}\}
\end{equation}
With $\delta \sim h$ we use the notation $\mcK_{h,x}$ and $\mcT_{h,x}$.
For each $\mcT_h$, $h \in (0,h_0]$ there is a set of points 
$\mathcal{X}_h$ on $\Gamma$ such that  
$\{\mcK_{h,x}, x \in \mathcal{X}_h\}$ and 
$\{\mcT_{h,x}, x \in \mathcal{X}_h\}$ are coverings of $\mcT_h$ and 
$\mcK_h$ with the following properties:
\begin{itemize}
\item The number of sets containing a given point $y$ is uniformly 
bounded
\begin{equation}
\#\{x\in \mathcal{X}_h:y\in \mcT_{h,x}\}\lesssim 1\quad \forall y\in \IR^3
\end{equation}
for all $h\in (0,h_0]$ with $h_0$ small enough.
\item The number of elements in the sets $\mcT_{h,x}$ 
is uniformly bounded
\begin{equation}
\# \mcT_{h,x} \lesssim 1\quad  \forall x \in \mathcal{X}_h
\end{equation}
for all $h\in (0,h_0]$ with $h_0$ small enough, and each element in $\mcT_{h,x}$ share at least one face with another 
element in $\mcT_{h,x}$.
%\item There is $\delta \sim h$ such that 
%\begin{equation}
%\mcT_{h,x}\subset B_\delta(x), \quad \mcK^l_{h,x}\subset D_\delta(x)
%\quad  \forall x \in \mathcal{X}_h, 
%\forall h \in (0,h_0]
%\end{equation}
\item $\forall h \in (0,h_0]$ and $\forall x \in \mathcal{X}_h$, 
 $\exists T_x \in \mcT_{h,x}$ that has a large 
intersection with $\Gammah$ in the sense that 
\begin{equation}
|T_x \cap \Gammah | = |K_x| \sim h^2\quad  
\forall x \in \mathcal{X}_h, 
\end{equation}
for all $h\in (0,h_0]$ with $h_0$ small enough.
\end{itemize}
Furthermore, there is a constant vector $\beta_{h,x}$ such that 
\begin{equation}\label{eq:betahx}
\| \beta_h - \beta_{h,x}\|_{L^\infty(\mcK_{h,x})} \lesssim h
\end{equation}
We first recall a Lemma from \cite{BuHaLa14} and then we prove 
two lemmas tailored to the particular demands of this paper.

\begin{lem}\label{lem:TechnicalA} It holds
\begin{equation}
\|v\|^2_{\mcT_h}
\lesssim
h\Big( \|v\|^2_{\mcK_h} +  \tn v \tn^2_{\mcF_h} \Big)
\qquad
\forall v \in V_h
\end{equation}
for all $h\in (0,h_0]$ with $h_0$ small enough.
\end{lem}
\begin{proof}
See Lemma 4.5 in \cite{BuHaLa14}.
\end{proof}

\begin{lem}\label{lem:TechnicalAA} It holds
\begin{equation}
h\| v \|^2_{\mcE_h}
\lesssim
\|v\|^2_{\mcK_h} +  h^2 \tn v \tn^2_{\mcF_h}
\qquad
\forall v \in V_h
\end{equation}
for all $h\in (0,h_0]$ with $h_0$ small enough.
\end{lem}
\begin{proof} Consider an arbitrary set in the covering 
described above. Then we shall prove that we have the 
estimate
\begin{equation}
h\| v \|^2_{\mcE_{h,x}} \lesssim \| v \|^2_{\mcK_{h,x}} 
+ h^2 \| [n_F \cdot \nabla v] \|^2_{\mcF_{h,x}}
\end{equation}
Let $v_x : \mcT_{h,x} \rightarrow \IR$ 
be the first order polynomial that satisfies $v_x = v|_{T_x}$, where $T_x$ is the element with a large intersection 
$K_x$. Adding and subtracting $v_x$ we get
\begin{equation}
h\| v \|^2_{\mcE_{h,x}} \leq h \| v - v_x \|^2_{\mcE_{h,x}} 
+ h \| v_x \|^2_{\mcE_{h,x}} = I + II
\end{equation}
\paragraph{Term $\bfI$.} We have 
\begin{equation}
h \| v - v_x \|^2_{\mcE_{h,x}} 
\lesssim 
h^{-1}  \| v - v_x \|^2_{\mcT_{h,x}}
\lesssim 
h^2 \tn v \tn^2_{\mcF_{h,x}}
\end{equation}
where we used the inverse estimates (\ref{eq:inverseEtoF}) 
and (\ref{eq:inverseFtoT}) to pass from $\mcE_h$ to 
$\mcT_h$, the inequality 
\begin{equation}\label{eq:L2jumpest}
\| w \|^2_{\mcT_{h,x}} 
\lesssim 
\| w \|^2_{T_x} + h^3 \tn w \tn^2_{\mcF_{h,x}}
\quad \forall w \in P_1(\mcT_{h,x})
\end{equation}
with $w=v-v_x=0$ on $T_{x}$, and finally the fact that
$[n_E \cdot \nabla v_x] = 0$. 

\paragraph{Verification of (\ref{eq:L2jumpest}).} Considering 
a pair of elements $T_1,T_2\in \mcT_{h,x}$ that share a 
face $F$ we have the identity 
\begin{equation}
w_2(x) = w_1(x) + [n_F\cdot \nabla w](x - x_F)\cdot n_F
\end{equation}
for any continuous piecewise linear polynomial $w$ with 
$w_i = w|_{K_i}$, $i=1,2.$ Integrating over $K_2$ gives 
\begin{equation}
\|w_2\|_{K_2}^2 \lesssim 
\|w_1\|_{K_2}^2 + \|[n_F\cdot \nabla w](x - x_F)\cdot n_F\|_{K_2}^2
\lesssim \|w_1\|_{K_1}^2 + h^3\|[n_F\cdot \nabla w]\|_{F}^2
\end{equation} 
Iterating this inequality and using the fact that the number of 
elements in $\mcT_{h,x}$ is uniformly bounded (\ref{eq:L2jumpest}) 
follows.

\paragraph{Term $\bfI\bfI$.}  We first split $v_x$ into 
one term $v_{x,c}$ which is constant in the 
direction normal to $K_x$ and a reminder term 
$v - v_x = t n_x \cdot \nabla v_x$ 
where $n_x = n_h|_{K_x}$ and $t$ is the signed distance 
to the hyperplane in which $K_x$ is contained, as follows
\begin{equation}
v_x = v_{x,c} + t n_x \cdot \nabla v_x
\end{equation} 
Using the triangle inequality we obtain 
\begin{align}
h \| v_x \|^2_{\mcE_{h,x}} 
&\lesssim
h \| v_{x,c} \|^2_{\mcE_{h,x}}
+ 
h \| t n_x \cdot \nabla v \|^2_{\mcE_{h,x}}
= II_1 + II_2
\end{align}

\noindent{\emph{Term $II_1$.}} We have
\begin{equation}
II_1  
= h \| v_{x,c} \|^2_{\mcE_{h,x}}
\lesssim  h^{-1} \| v_{x,c} \|^2_{\mcT_{h,x}}
\lesssim h^{-1} \| v_{x,c} \|^2_{T_{x}}
\lesssim \| v_{x,c} \|^2_{K_{x}}
\lesssim \| v \|^2_{\mcK_{h,x}}
\end{equation}
where we used the inverse estimates (\ref{eq:inverseEtoF}) 
and (\ref{eq:inverseFtoT}) to pass from $\mcE_h$ to 
$\mcT_h$, an inverse estimate using the fact that $v_{x,c}$ 
is a polynomial on $\mcT_{h,x}$, finally an inverse inequality 
which holds since $v_{x,c}$ is constant in the normal direction.

\noindent{\emph{Term $II_2$.}} We have
\begin{align}
II_2 
&\lesssim h \|t\|^2_{L^\infty(\mcE_{h,x})} 
\| n_E \cdot \nabla v_x \|^2_{\mcE_{h,x}}
\lesssim h^5 \| n_E \cdot \nabla v_x \|^2_{\mcE_{h,x}}
\\ \nonumber
&\qquad \lesssim h^3 \| \nabla v_x \|^2_{T_{x}}
\lesssim h \| v_x \|^2_{T_{x}}
\lesssim h \| v \|^2_{\mcT_{h,x}}
\end{align}
where we used H\"older's inequality, the bound 
\begin{equation}\label{eq:tbound}
\| t \|_{L^\infty(\mcE_{h,x})} \lesssim h^2
\end{equation}
the inverse estimates (\ref{eq:inverseEtoF}) 
and (\ref{eq:inverseFtoT}) to pass from $\mcE_h$ to 
$\mcT_h$, an inverse inequality to pass from $H^1$ 
to $L^2$.

\paragraph{Verification of (\ref{eq:tbound}).} We note 
that each point $y \in \mcK_{h,x}$ can be connected to 
a point $z\in K_x$ using a piecewise linear curve in 
$\mcK_{h,x}$ consisting of a finite number of segments, 
each residing in a facet $K_i \in \mcK_{h,x}$, of the form 
\begin{equation}
y = z + \sum_{i=1}^N s_i a_i  
\end{equation}
where $0\leq s_i \lesssim h$ is the arclength parameter of each 
segment, $a_i \in T(K_i)$ is a unit direction vector in the tangent space $T(K_i)$ to $K_i$, and $N$ is uniformly bounded. 
Then $t(y)= n_x \cdot (y - z )$ and we have the estimate 
\begin{align}
|t(y)| 
\leq \sum_{i=1}^N |s_i|\, |n_x \cdot a_i|
\leq \sum_{i=1}^N |s_i|\, |(n_x-n_{h,i})\cdot a_i|
\leq \sum_{i=1}^N |s_i|\, |n_x-n_{h,i}|
\lesssim h^2
\end{align}
where $n_{h,i}$ is the normal to the facet in which the 
$i$:th segment reside and we used the estimate $|n_x-n_{h,i}|\lesssim h$, which follows from the fact that at each edge $E$ 
shared by two facets $K_1,K_2\in \mcK_{h,x}$ with normals $n_{h,1}$ and $n_{h,2}$ we have the estimate
\begin{equation}
|n_{h,1} - n_{h,2}| \leq |n_{h,1} - n| + |n - n_{h,2}|\lesssim h
\end{equation}
and thus we obtain the bound since the number of elements in 
$\mcK_{h,x}$ is uniformly bounded.

%Let $p_{h,x} 
%= I-n_x\otimes n_x
%= P_{K_x}$ and $\rho_{h,x}$ denote the closest point 
%projection and the signed distance function to the 
%hyperplane containing $K_x$.
%Then $t=\rho_{h,x}$ on $\mcK_{h,x}$ and 
%we have the estimate
%\begin{align}
%\| t \|_{L^\infty(\mcK_{h,x})} 
%&= 
%\| y - p_{h,x}(y)\|_{L^\infty(\mcK_{h,x})}
%\\
%&\leq 
%\| y - p(y)\|_{L^\infty(\mcK_{h,x})}
%\\ \nonumber
%&\qquad +
%\| p(y) - P_{\Gamma}(z)y\|_{L^\infty(\mcK_{h,x})}
%+
%\| P_{\Gamma}(z)y - p_{h,x}(y) \|_{L^\infty(\mcK_{h,x})}
%\end{align} 
%where $z$ is some point in $p(K_x)\subset \Gamma$. Here the 
%first term satisfies $\| y - p_{h,x}(y)\|_{L^\infty(\mcK_{h,x})}
%\lesssim h^2$ according to (\ref{assum:geom}). The third is estimated as follows
%\begin{equation}
%\| P_{\Gamma}(z)y - p_{h,x}(y) \|_{L^\infty(\mcK_{h,x})}
%\leq |n(z)\otimes n(z) - n_x \otimes n_x|\,|y|
%\lesssim h^2
%\end{equation}

\paragraph{Conclusion.} Collecting the estimates 
we obtain
\begin{equation}
h\| v \|^2_{\mcE_{h,x}} 
\lesssim I + II_1 + II_2 
\lesssim \| v \|^2_{\mcK_{h,x}}
+ 
h^2 \tn v \tn^2_{\mcF_{h,x}}
+
h \| v \|^2_{\mcT_{h,x}}
\end{equation}
Summing over the covering and using Lemma \ref{lem:TechnicalA} gives
\begin{align}
h\| v \|^2_{\mcE_{h}} 
&\lesssim 
\| v \|^2_{\mcK_{h}}
+ 
h^2 \tn v \tn^2_{\mcF_{h}}
+
h \| v \|^2_{\mcT_{h}}
\\
&\lesssim 
\| v \|^2_{\mcK_{h}}
+ 
h^2 \tn v \tn^2_{\mcF_{h}}
+ 
h^2\Big( \| v \|_{\mcK_h}^2 + \tn v \tn^2_{\mcF_h} \Big)
\\  
&\lesssim 
\| v \|^2_{\mcK_{h}}
+ 
h^2 \tn v \tn^2_{\mcF_{h}}
\end{align}
which concludes the proof.
\end{proof}

\begin{lem}\label{lem:TechnicalB} It holds
\begin{equation}
\|(I-\pi_h)(\beta_h \cdot \nablash v)\|^2_{\mcK_h}
\lesssim
\|v\|^2_{\mcK_h} +  \tn v \tn^2_{\mcF_h}
\qquad
\forall v \in V_h
\end{equation}
for all $h\in (0,h_0]$ with $h_0$ small enough.
\end{lem}
\begin{proof}
We use a covering $\{\mcT_{h,x} : x \in \mathcal{X}_h\}$ as described 
above. For each set $\mcT_{h,x}$ of elements in the covering 
we have 
\begin{align}
\nonumber
&h \|(I-\pi_h)(\beta_h\cdot \nabla_{\Gammah} v)\|^2_{\mcK_{h,x}} 
\\ \label{eq:TechBa}
&\qquad \lesssim \|(I-\pi_h)(\beta_h\cdot \nabla_{\Gammah} v)\|^2_{\mcT_{h,x}}
\\ \label{eq:TechBb}
&\qquad \lesssim \|(I-\pi_h)(\beta_{h,x}\cdot \nabla v)\|^2_{\mcT_{h,x}}
+ \|(I-\pi_h)((\beta_h - \beta_{h,x})\cdot \nabla v)\|^2_{\mcT_{h,x}}
\\ \label{eq:TechBc}
&\qquad \lesssim h \|[n_F\cdot\nabla v]\|^2_{\mcF_{h,x}}
+ \|\beta_h - \beta_{h,x}\|^2_{L^\infty(\mcT_{h,x})} 
\| \nabla v \|^2_{\mcT_{h,x}}
\\ \label{eq:TechBd}
&\qquad \lesssim h \|[n_F\cdot\nabla v]\|^2_{\mcF_{h,x}}
+ \| v \|^2_{\mcT_{h,x}}
\end{align}
Here we used the following estimates. (\ref{eq:TechBa}): An inverse 
bound to pass from $\mcK_{h,x}$ to $\mcT_{h,x}$. (\ref{eq:TechBb}): 
Added and subtracted $\beta_{h,x}$ and used the triangle inequality.
(\ref{eq:TechBc}): The second term is directly estimated and for 
the first term we used the following Poincar\'e inequality  
\begin{equation}\label{eq:TechBe}
\|(I-\pi_h) w \|_{\mcT_{h,x}}^2 \lesssim h \|[w]\|^2_{\mcF_{h,x}} 
\qquad \forall w \in DP_0(\mcT_{h,x}) 
\end{equation}
where $DP_0(\mcT_{h,x})$ is the space of piecewise constant functions 
on $\mcT_{h,x}$. (\ref{eq:TechBd}): 
The bound (\ref{eq:betahx}) followed by an inverse estimate to remove the gradient.

\paragraph{Verification of (\ref{eq:TechBe}).}
We first map $\mcT_{h,x}$ to a reference configuration $\widehat{\mcT}_{h,x}$ and then apply a 
compactness argument based on the observation that if the right hand 
side is zero then $w$ must be constant on $\mcT_{h,x}$ but then also 
the left hand side is zero due to the interpolation operator $I-\pi_h$. 
Thus the inequality hold on on the reference configuration and we 
map back $\mcT_{h,x}$. Finally, we note that due to quasiuniformity 
there is a uniform bound on the number reference configurations necessary
since the number of elements in $\mcT_{h,x}$ is uniformly bounded

\paragraph{Conclusion.} Finally, summing over the covering sets 
and using Lemma 
\ref{lem:TechnicalA} we obtain
\begin{align}
h \|(I-\pi_h)(\beta_h\cdot \nabla_{\Gammah} v)\|^2_{\mcK_{h}}
&\lesssim \| v \|^2_{\mcT_{h}} + h \|[n_F\cdot\nabla v]\|^2_{\mcF_{h}}
\\
&\lesssim h \| v \|^2_{\mcK_{h}} + 
h \|[n_F\cdot\nabla u]\|^2_{\mcF_h}
\end{align}
which concludes the proof.
\end{proof}
\begin{lem}\label{lem:TechnicalC} It holds
\begin{equation}
\| \beta_h \cdot \nablash v\|^2_{\mcT_h}
\lesssim
 h \| v \|^2_{\mcK_h} + 
 h \| \beta_h \cdot \nablash v\|^2_{\mcK_h} + h \tn v \tn^2_{\mcF_h}
 \qquad \forall v \in V_h
\end{equation}
for all $h\in (0,h_0]$ with $h_0$ small enough.
\end{lem}
\begin{proof} We again consider an arbitrary set $\mcT_{h,x}$ in the covering. Adding and subtracting $\beta_{h,x}$, that satisfies the 
estimate (\ref{eq:betahx}), and using the triangle inequality we get
\begin{align}
\| \beta_h\cdot \nablash v \|^2_{\mcT_{h,x}} 
&\leq \| \beta_{h,x} \cdot \nablash v \|^2_{\mcT_{h,x}}
+ \| (\beta_h - \beta_{h,x}) \cdot \nablash v \|^2_{\mcT_{h,x}}
\\ \label{eq:TechCa}
&\lesssim \Big( h \| \beta_{h,x} \cdot \nablash v \|^2_{K_{x}}
+ h \tn v \tn^2_{\mcF_{h,x}} \Big) 
+ h^2 \| \nabla v \|^2_{\mcT_{h,x}}
\\ \label{eq:TechCb}
&\lesssim \Big( h \| \beta_{h,x} \cdot \nablash v \|^2_{\mcK_{h,x}}
+ h \tn v \tn^2_{\mcF_{h,x}} \Big) 
+ \| v \|^2_{\mcT_{h,x}}
\end{align}
Here we used the following estimates: (\ref{eq:TechCa}): Again 
using a compactness argument we conclude that
\begin{align}
\|w\|^2_{\mcT_{h,x}} &\lesssim \| w \|^2_{T} 
+ h\|[w]\|^2_{\mcF_{h,x}} \qquad \forall w \in DP_0(\mcT_{h,x})
\end{align}
for any $T \in \mcT_{h,x}$. Now taking $T=T_{x}$, the 
element with a large intersection $T\cap \Gammah$ we also 
have the inverse estimate 
$\| w \|_{T_{x}}^2 \lesssim h \| w \|^2_{K_{x}}$ 
since $w$ is constant on $T$. 
(\ref{eq:TechCb}): Follows directly from the fact that 
$K_{x} \in {\mcK_{h,x}}$.

Summing over the sets in the covering gives
\begin{align}
\| \beta_h\cdot \nablash v \|^2_{\mcT_{h}} 
\lesssim h \| \beta_h \cdot \nablash v \|^2_{\mcK_h} 
+ h \tn v \tn^2_{\mcF_{h}}
+ \| v \|^2_{\mcT_h}
\end{align}
and using Lemma \ref{lem:TechnicalA} we can bound the last 
term and arrive at 
\begin{align}
\| \beta_h\cdot \nablash v \|^2_{\mcT_{h}} 
\lesssim h \| \beta_h \cdot \nablash v \|^2_{\mcK_h} 
+ h \tn v \tn^2_{\mcF_{h}}
+ h \| v \|^2_{\mcK_h}
\end{align}
which concludes the proof.
\end{proof}

\subsection{Stability Estimates}
\paragraph{Assumption.} We assume that the discrete coefficients 
$\alpha_h$ and $\beta_h$ are such that:
\begin{equation}\label{eq:bethalphahassump}
0< C \leq \alpha_h(x) - \frac{1}{2}\divsh \beta_h(x) \quad \forall 
x \in \Gammah
\end{equation}
\begin{equation}\label{eq:bethedgeestassumption}
\|[n_E\cdot\beta_h]\|_{L^\infty(\mcE_h)} \lesssim h^2
\end{equation}
for all $h\in (0,h_0]$ with $h_0$ small enough.

We return to the construction of $\alpha_h$ and $\beta_h$ 
in Section \ref{Sec:QuadErrors}.

\begin{lem}\label{lem:staba} There is a positive 
constant $m_0$ such that 
\begin{equation}
m_0\Big( \| v \|_{\mcK_h}^2 + h \tn v \tn_{\mcF_h}^2\Big) 
\leq A_h(v,v)\qquad \forall v \in V_h
\end{equation}
for all $h \in (0,h_0]$ with $h_0$ small enough.
\end{lem}
\begin{proof}
Integrating by parts elementwise over the surface mesh 
$\mcK_h$ we obtain the identity
\begin{equation}\label{eq:stabaa}
2 (\beta_h \cdot \nablash v,v)_{\mcK_h} = ([n_E \cdot \beta_h] v,v)_{\mcE_h} 
- (\nablash\cdot\beta_h)v,v)_{\mcK_h}
\end{equation}
The first term on the right hand side may be estimated 
as follows
\begin{align}
([n_E \cdot \beta_h] v,v)_{\mcE_h} 
&= \| [n_E \cdot \beta_h] \|_{L^\infty(\mcE_h)} \|v\|^2_{\mcE_h}
\\
&\lesssim h^2 \| v \|^2_{\mcE_h}
\\
&\lesssim h\| v \|^2_{\mcK_h} + h^3 \tn v \tn^2_{\mcF_h}
\end{align}
where we used H\"older's inequality, the assumption 
(\ref{eq:bethedgeestassumption}) on $\beta_h$, and 
at last Lemma \ref{lem:TechnicalAA}. Thus we arrive at 
the estimate 
\begin{equation}\label{eq:stabab}
|([n_E \cdot \beta_h] v,v)_{\mcE_h}| 
\leq C_\star h \Big( \| v \|^2_{\mcK_h} 
+ h^2 \tn v \tn^2_{\mcF_h} \Big)
\end{equation}
We now have
\begin{align}
A_h(v,v) &= (\beta_h \cdot \nablash v, v)_\Gammah 
+ (\alpha v,v)_\Gammah 
+ (c_F h [n_F \cdot \nabla v],[n_F \cdot \nabla v])_{\mcF_h}
\\
&=((\alpha - 2^{-1}\divsh \beta_h) v,v)_\Gammah 
\\ \nonumber
&\qquad + 2^{-1}([n_E \cdot \beta_h] v,v)_{\mcE_h} 
+ (c_F h [n_F \cdot \nabla v],[n_F \cdot \nabla v])_{\mcF_h}
\\
&\geq \inf_{\mcK_h}(\alpha - 2^{-1} \divsh \beta_h)
\|v\|_{\mcK_h}^2
\\ \nonumber
&\qquad - 2^{-1} C_\star h \Big( \| v \|^2_{\mcK_h} 
+ h^2 \tn v \tn^2_{\mcF_h} \Big) 
+ c_F h \tn v \tn^2_{\mcF_h}
\\
&\geq \inf(\alpha - 2^{-1} \divsh \beta_h 
- 2^{-1} C_\star h ) \|v\|^2_{\mcK_h}
\\ \nonumber
&\qquad + \min(c_F- 2^{-1}C_\star h^2) h\tn v \tn^2_{\mcF_h}
\end{align}
where we used the identity (\ref{eq:stabaa}), the estimate 
(\ref{eq:stabab}), and then we collected the terms. Thus we 
find that 
\begin{equation}
\| v \|_{\mcK_h}^2 + h\|[n_F\cdot \nabla v]\|_{\mcF_h}^2 
\lesssim A_h(v,v) 
\end{equation}
for $c_F>0$ and $h \in (0,h_0]$ with $h_0$ small enough.
\end{proof}

\begin{lem}\label{lem:stabb} There are positive constants 
$m_1$ and $m_2$ such that 
\begin{equation}
m_1 h\|\beta_h \cdot \nablas v \|^2_{\mcK_h} 
- m_2 A_h(v,v)
\leq 
A_h(v,h \pi_h(\beta_h \cdot \nablas v)) 
\qquad 
v \in V_h
\end{equation}
for all $h \in (0,h_0]$ with $h_0$ small enough.
\end{lem}
\begin{proof} We have
\begin{align}
A_h(v,h\pi_h(\beta_h \cdot \nabla v)) 
&= (\beta_h \cdot \nablash v, h \beta_h \cdot \nablash v)_{\mcK_h} 
\\ \nonumber
&\qquad\underbrace{- (\beta_h \cdot \nablash v, h (I - \pi_h) (\beta_h \cdot \nablash v))_{\mcK_h}}_{I} 
\\ \nonumber
%&\qquad - \frac{1}{2}([n_E\cdot \beta_h]u, h \pi_h (\beta_h \cdot %\nablash u))_{\mcE_h}
%\\ \nonumber
&\qquad \underbrace{+  (\alpha u, h \pi_h (\beta_h \cdot \nablash u))_{\mcK_h}}_{II}
\\ \nonumber
&\qquad
\underbrace{+j_h(v,h \pi_h (\beta_h \cdot \nablash v)}_{III}
%\\ \label{eq:stabba}
%&=h\| \beta_h \cdot \nablash u \|^2_{\mcK_h} 
%+ I + II + III 
\\ \label{eq:stabbbb}
&\geq h\| \beta_h \cdot \nablash v \|^2_{\mcK_h} 
- |I + II + III| 
\end{align}
We now have the estimate
\begin{equation}\label{eq:stabbc}
|I+II+III|\leq C_1 (\delta + \delta^{-1}) A_h(v,v) 
+ C_2 \delta  h \|\beta_h \cdot \nablash v \|^2_{\mcK_h}
\end{equation}
for $\delta>0$. Thus taking $\delta$ small enough the desired 
estimate follows directly by combining (\ref{eq:stabbbb}) and 
(\ref{eq:stabbc}).

\paragraph{Verification of (\ref{eq:stabbc}).} The estimate follows by combining the following estimates of Terms $I-III$.

\paragraph{Term $\bfI$.} It holds 
\begin{align}
|I| &\lesssim \delta h \|  \beta_h \cdot \nablash v \|^2_{\mcK_h} 
+ \delta^{-1} h \| (I-\pi_h) \beta_h \cdot \nablash v \|^2_{\mcK_h}
\\
&\lesssim \delta h \|  \beta_h \cdot \nablash v \|^2_{\mcK_h} 
+ \delta^{-1} h \Big( \| v \|^2_{\mcK_h} + \tn v \tn^2_{\mcF_h} \Big)
\end{align}
where we used the inequality Cauchy-Schwarz, the 
arithmetic-geometric mean inequality with parameter $\delta>0$, 
and Lemma \ref{lem:TechnicalB}.

\paragraph{Term $\bfI\bfI$.} It holds
\begin{align}
|II|&\lesssim \|\alpha\|_{L^{\infty}(\mcK_h)} \|v \|_{\mcK_h}  
h\|\pi_h (\beta_h \cdot \nablash v)\|_{\mcK_h}
\\
&\lesssim \delta^{-1} h \| v \|_{\mcK_h}^2  
+ \delta h \|\pi_h (\beta_h \cdot \nablash v)\|^2_{\mcK_h}
\\
&\lesssim \delta^{-1} h \| v \|_{\mcK_h}^2  
+ \delta \|\pi_h (\beta_h \cdot \nablash v)\|^2_{\mcT_h}
\\
&\lesssim \delta^{-1} h \| v \|_{\mcK_h}^2  
+ \delta \| \beta_h \cdot \nablash v\|^2_{\mcT_h}
\\
&\lesssim \delta^{-1} h \| v \|_{\mcK_h}^2  
+ \delta h \Big(\| v \|^2_{\mcK_h} + 
\|\beta_h \cdot \nablash v \|^2_{\mcK_h} + \tn v \tn^2_{\mcF_h} \Big)
\\
&\lesssim (\delta^{-1} + \delta) h \| v \|_{\mcK_h}^2  
+ \delta h \tn v \tn^2_{\mcF_h} + \delta h \|\beta_h \cdot \nablash v \|^2_{\mcK_h}
\end{align}
where we used H\"older's inequality, the arithmetic-geometric 
mean inequality with parameter $\delta>0$, the inverse estimate 
(\ref{eq:inverseKtoT}) to pass from $\mcK_h$ to $\mcT_h$, the boundedness (\ref{eq:pihl2stab}) of $\pi_h$ on $L^2(\mcT_h)$, 
Lemma \ref{lem:TechnicalC}, and finally we rearranged the terms.

\paragraph{Term $\bfI\bfI\bfI$.} It holds
\begin{align}
|III| &\lesssim h \tn v \tn_{\mcF_h} 
\tn h\pi_h(\beta_h \cdot \nabla v) \tn_{\mcF_h}
\\
&\lesssim \delta^{-1} h \tn v \tn_{\mcF_h}^2 
+ \delta h^3 \tn \pi_h(\beta_h \cdot \nabla v) \tn_{\mcF_h}^2
\\
&\lesssim \delta^{-1} h \tn v \tn_{\mcF_h}^2 
+ \delta h^2 \| \nabla \pi_h(\beta_h \cdot \nabla v) \|_{\mcT_h}^2
\\
&\lesssim \delta^{-1} h \tn v \tn_{\mcF_h}^2 
+ \delta \| \pi_h(\beta_h \cdot \nabla v) \|_{\mcT_h}^2
\\
&\lesssim \delta^{-1} h \tn v \tn_{\mcF_h}^2 
+ \delta \| \beta_h \cdot \nabla v \|_{\mcT_h}^2
\\
&\lesssim \delta^{-1} h \tn  v \tn_{\mcF_h}^2 
+ \delta h \Big(\| v \|^2_{\mcK_h} + 
\|\beta_h \cdot \nablash v \|^2_{\mcK_h} + \tn v \tn^2_{\mcF_h} \Big)
\\
&\lesssim  \delta h \| v \|_{\mcK_h}^2  
+ (\delta^{-1} + \delta) h \tn v \tn^2_{\mcF_h} + \delta h \|\beta_h \cdot \nablash v \|^2_{\mcK_h}
\end{align}
where we used the Cauchy-Schwarz inequality, the arithmetic-geometric mean
inequality with parameter $\delta>0$, the inverse estimate (\ref{eq:inverseFtoT}) to pass from $\mcF_h$ to $\mcT_h$, an inverse inequality to remove the gradient, the boundedness (\ref{eq:pihl2stab}) of $\pi_h$ on $L^2(\mcT_h)$, 
Lemma \ref{lem:TechnicalC}, and finally we rearranged the terms.
\end{proof}

\begin{prop}\label{prop:stab} There is a positive constant 
$m_3$ such that 
\begin{equation}\label{eq:stabest}
m_3 \tn  v \tn_h
\leq \sup_{w \in V_h \setminus \{ 0\}} \frac{A_h(v,w)}{ \tn w \tn_h}
\qquad  \forall v \in V_h 
\end{equation}
for all $h\in (0,h_0]$ with $h_0$ small enough.
\end{prop}
\begin{proof}
Setting $w=v + \gamma h \pi_h(\beta_h \cdot \nablash v)$, for some 
positive parameter $\gamma$, we get
\begin{align}
A_h(v,w)&=A(v,v) + \gamma A(v,h \pi_h(\beta_h \cdot \nablash v))
\\
&\geq A(v,v) + \gamma m_1 h \| \beta_h \cdot \nablash v  \|^2_{\mcK_h}
- \gamma m_2 A(v,v)
\\
&=  (1-\gamma m_2 ) A(v,v) 
+ \gamma m_1 h \| \beta_h \cdot \nablash v \|^2_{\mcK_h}
\\
&=  (1-\gamma m_2 ) m_0^{-1}(\| v \|_{\mcK_h}^2 
+ h\|[n_F \cdot \nabla v]\|_{\mcF_h}^2) 
+ \gamma m_1 h \| \beta_h \cdot \nablash v \|^2_{\mcK_h}
\\
&\geq \widetilde{m}_3 \tn v  \tn_h^2
\end{align}
where we used Lemma \ref{lem:staba} and choose $0<\gamma$ small enough. 
Using the bound 
\begin{equation}\label{eq:stabesta}
\tn w \tn_h \leq C \tn v \tn_h
\end{equation}
the desired estimate follows with $m_3 = \widetilde{m}_3/C$.

\paragraph{Verification of (\ref{eq:stabesta}).} Using the triangle inequality 
\begin{align}\label{eq:stabestb}
\tn v + \gamma \pi_h( \beta_h \cdot \nablash v ) \tn_h^2
&\lesssim \tn v \tn_h^2  + \gamma^2 \tn h \pi_h( \beta_h \cdot \nablash v ) \tn_h^2
\end{align}  
where the second term takes the form
\begin{align}
\tn h \pi_h( \beta_h \cdot \nablash v ) \tn_h^2
&=h^2 \|\pi_h( \beta_h \cdot \nablash v )\|^2_{\mcK_h} 
\\ \nonumber
&\qquad + h^3 \| \beta_h\cdot \nablash \pi_h( \beta_h \cdot \nabla v )\|^2_{\mcK_h}
+ h^3 \tn \pi_h( \beta_h \cdot \nablash v )\tn^2_{\mcF_h}
\\
&=I+ II + III
\end{align}
\paragraph{Term $\bfI$.} It holds
\begin{align}
I&= h^2 \|\pi_h( \beta_h \cdot \nablash v )\|^2_{\mcK_h} 
\\
&\lesssim h \|\pi_h( \beta_h \cdot \nablash v )\|^2_{\mcT_h}
\\
&\lesssim h \| \beta_h \cdot \nablash v \|^2_{\mcT_h}
\\
&\lesssim h^2 \|v\|^2_{\mcK_h}
+ h^2 \|\beta_h \cdot \nablash v\|^2_{\mcK_h} 
+ h^2 \tn v \tn^2_{\mcF_h}
\end{align}
where we used the inverse estimate (\ref{eq:inverseKtoT}) to 
pass from $\mcK_h$ to $\mcT_h$, the boundedness (\ref{eq:pihl2stab}) of $\pi_h$ on $L^2(\mcT_h)$, and at last Lemma \ref{lem:TechnicalC}.

\paragraph{Term $\bfI\bfI$.} It holds
\begin{align}
h^3 \| \beta_h\cdot \nablash \pi_h( \beta_h \cdot \nablash v )\|^2_{\mcK_h}
&\lesssim
h^2 \| \beta_h \cdot \nablash \pi_h( \beta_h \cdot \nablash v )\|^2_{\mcT_h}
\\
&\lesssim \| \pi_h( \beta_h \cdot \nablash v )\|^2_{\mcT_h}
\\
&\lesssim \| \beta_h \cdot \nablash v \|^2_{\mcT_h}
\\
&\lesssim h\|v\|^2_{\mcK_h}
+ h \|\beta_h \cdot \nablash v\|^2_{\mcK_h} 
+ h \tn v \tn^2_{\mcF_h}
\end{align}
where used the inverse estimate (\ref{eq:inverseKtoT}) 
to pass from $\mcK_h$ to $\mcT_h$, an inverse estimate 
to remove the transport derivative, the boundedness (\ref{eq:pihl2stab}) of 
$\pi_h$ on $L^2(\mcT_h)$, and finally we used Lemma \ref{lem:TechnicalC}. 

\paragraph{Term $\bfI\bfI\bfI$.} It holds
\begin{align} 
h^3 \tn \pi_h(\beta_h \cdot \nablash v ) \tn^2_{\mcF_h} 
& \lesssim 
h^2 \| \nabla \pi_h( \beta_h \cdot \nablash v )\|^2_{\mcT_h}
\\
& \lesssim
\| \pi_h( \beta_h \cdot \nablash v )\|^2_{\mcT_h}
\\
&\lesssim \| \beta_h \cdot \nablash v \|^2_{\mcT_h}
\\
&\lesssim h\|v\|^2_{\mcK_h}
+ h \|\beta_h \cdot \nablash v\|^2_{\mcK_h} 
+ h \tn v \tn^2_{\mcF_h}
\end{align}
where we used the inverse inequality (\ref{eq:inverseFtoT}) 
to pass from $\mcF_h$ to $\mcT_h$, an inverse estimate to 
remove the gradient, the boundedness (\ref{eq:pihl2stab}) of $\pi_h$ on 
$L^2(\mcT_h)$, and finally we used Lemma \ref{lem:TechnicalC}. 

\paragraph{Conclusion of Verification of (\ref{eq:stabesta}).}
Combining the estimates of Terms $I-III$ we get 
\begin{equation}
\tn h \pi_h( \beta_h \cdot \nablash v ) \tn_h^2 
\lesssim h\|v\|^2_{\mcK_h}
+ h \|\beta_h \cdot \nablash v\|^2_{\mcK_h} 
+ h \tn v \tn^2_{\mcF_h}
\lesssim \tn v \tn_h^2
\end{equation}
and therefore, in view of (\ref{eq:stabestb}), we conclude 
that (\ref{eq:stabesta}) holds.
\end{proof}

\begin{prop}\label{prop:stabB} It holds
\begin{equation}\label{eq:stabB}
h^{3/4} \| \nablash v \|_{\mcK_h} 
%\lesssim h\|v\|^2_{\mcK_h} + h \tn v \tn^2_{\mcF_h} 
\lesssim \tn v \tn_h \qquad \forall v \in V_h 
\end{equation}
for all $h\in (0,h_0]$ with $h_0$ small enough.
\end{prop}
\begin{proof} Using partial integration followed 
by Cauchy-Schwarz we have
\begin{align}
\| \nablash v \|^2_{\mcK_h} &= (\nablash v, \nablash v)_{\mcK_h}
\\
&=-(v,[n_E \cdot \nabla  v])_{\mcE_h}
\\  \label{eq:propstabBprod}
&\leq \underbrace{\|v\|_{\mcE_h}}_{I} \underbrace{\|[n_E \cdot \nabla  v]\|_{\mcE_h}}_{II}
\end{align}
\paragraph{Term $\bfI$.} 
Using Lemma \ref{lem:TechnicalAA} we directly obtain
\begin{equation}
\| v \|^2_{\mcE_h} 
\lesssim 
h^{-1}\Big( \| v \|^2_{\mcK_h} 
+
h^2 \tn v \tn^2_{\mcF_h}\Big) 
\lesssim
h^{-1} \| v \|^2_{\mcK_h} 
+
h \tn v \tn_{\mcF_h}^2
\end{equation}
and we conclude that 
\begin{equation}\label{eq:propstabBtermI}
h \| v \|^2_{\mcE_h} \lesssim \| v \|^2_{\mcK_h} 
+
h^2 \tn v \tn^2_{\mcF_h}
\lesssim \tn v \tn_h^2
\end{equation}

\paragraph{Term $\bfI\bfI$.} We have the estimates
\begin{align}
\|[n_E \cdot \nabla  v]\|^2_{\mcE_h} 
&\lesssim h^{-1}\|[n_E \cdot \nabla  v ]\|^2_{\mcF_h}
\\
&\lesssim h^{-1}\|[n_E] \cdot \langle \nabla v \rangle \|^2_{\mcF_h}
+ h^{-1} \|\langle n_E\rangle \cdot [\nabla  v ]\|^2_{\mcF_h}
\\
&\lesssim h^{-1}\| h \nabla  v \|^2_{\mcF_h}
+ h^{-1} \|\langle n_E\rangle \cdot n_F  [n_F \cdot \nabla v]\|^2_{\mcF_h}
\\
&\lesssim h^{-2}\| h \nabla  v \|_{\mcT_h}^2
+ h^{-1} \| [n_F \cdot \nabla  v]\|^2_{\mcF_h}
\\
&\lesssim h^{-2}\|  v \|_{\mcT_h}^2
+ h^{-1} \| [n_F \cdot \nabla  v]\|^2_{\mcF_h}
\\
&\lesssim h^{-2}\Big( h \|  v \|_{\mcK_h}^2 
+ h\| [n_F \cdot \nabla  v]\|^2_{\mcF_h} \Big) 
+ h^{-1} \| [n_F \cdot \nabla  v]\|^2_{\mcF_h}
\\
&\lesssim h^{-1}\Big( \|  v \|_{\mcK_h}^2 
+  \| [n_F \cdot \nabla  v]\|^2_{\mcF_h} \Big)
\end{align}
Here we used the inverse inequality (\ref{eq:inverseEtoF}) 
to pass from $\mcE_h$ to $\mcF_h$, the identity 
$[ab] = [a]\langle b \rangle + \langle a \rangle [b]$, 
where $\langle a \rangle = (a^+ + a^-)/2$ is the average 
of a discontinuous function, the fact that the tangent gradient 
is continuous at a face, the inverse inequality (\ref{eq:inverseFtoT}) to pass from $\mcF_h$ to $\mcT_h$ in 
the first term and a direct estimate for the second, 
Lemma \ref{lem:TechnicalA} for the first term, and finally 
we collected the terms. We conclude that 
\begin{equation}\label{eq:propstabBtermII}
h^2 \| [ n_E \cdot \nabla v]\|^2_{\mcE_h} \lesssim \tn v \tn_h^2
\end{equation}

\paragraph{Conclusion.} Combining (\ref{eq:propstabBprod}) with the estimates (\ref{eq:propstabBtermI}) and (\ref{eq:propstabBtermII}) 
we obtain
\begin{equation}
h^3\|\nablash v \|^4_{\mcK_h} 
\lesssim 
h\|v\|^2_{\mcE_h} h^2 \| [n_E \cdot \nablash v] \|^2_{\mcE_h}
\lesssim 
\tn v \tn_h^4
\end{equation}
which concludes the proof.
\end{proof}

\section{Error Estimates}
%\todo[inline]{Add in assumption on $\beta_h$ in a suitable place. 
%For the Strang Lemma we need equivalence of norms which requires 
%first order accuracy. Maybe put that result in a separate lemma}
\subsection{Strang's Lemma}

Define the forms
\begin{equation}
a(v,w) = (\beta \cdot \nabla v, w)_\Gamma + (\alpha u,v)_\Gamma, 
\qquad l(v) = (f,v)_\Gamma
\end{equation}
Then the exact solution $u$ to the convection problem 
(\ref{eq:conva}), see Proposition \ref{prop:existence}, satisfies 
\begin{equation}
a(u,v) = l(v)\qquad \forall v \in L^2(\Gamma) 
\end{equation}
We then have the following Strang Lemma.

\begin{lem}\label{lem:strang} Let $u$ be the solution to  \eqref{eq:conva} and $u_h$ the finite element approximation 
defined by \eqref{eq:fem}, then the following estimate holds 
\begin{align}\label{eq:strang}
\tn u^e-u_h \tn_h
&\lesssim h^{3/2}\| u \|_{H^2(\Gamma)}
\\ \nonumber 
&\qquad 
+ \sup_{v \in V_h \setminus 0} \frac{a((\pi_h u^e)^l,v^l) - a_h(\pi_h u^e,v)}{\tn v \tn_h}\nonumber
\\ \nonumber
&\qquad + \sup_{v \in V_h \setminus 0} \frac{l( v^l) - l_h(v)}{\tn v^l\tn_h}
\end{align} 
\end{lem}
\begin{proof} Adding and subtracting an interpolant $\pi_h u^e$, and then using the triangle inequality we obtain
\begin{align}
 \tn u^e - u_h \tn_h
 &\leq \tn u - \pi_h u^e \tn_h
+ \tn \pi_h u^e  - u_h \tn_h
\\
&\lesssim h^{3/2} \| u \|_{H^2(\Gamma)} 
+ \tn \pi_h u^e  - u_h \tn_h
\end{align}
where we used the interpolation estimate (\ref{eq:interpolenergysurf}) to estimate 
the first term. Proceeding with the second 
term we employ the inf-sup estimate in Proposition 
\ref{prop:stab} to get the bound 
\begin{equation}\label{eq:infsup}
\tn \pi_h u^e  - u_h \tn_h
\lesssim \sup_{v \in V_h \setminus \{0\}} 
\frac{A_h(\pi_h u  - u_h,v)}{\tn v \tn_h}
\end{equation}
Adding and subtracting the exact solution, and using 
Galerkin orthogonality (\ref{eq:fem}) the numerator may 
be written in the following form
\begin{align}
A_h( \pi_h u^e  - u_h,v) 
&= 
A_h( \pi_h u^e,v) -  l_h(v) 
\\
&=A_h( \pi_h u^e,v) - a((\pi_h u^e)^l , v^l) 
+ a((\pi_h u^e)^l - u,v^l) + l(v^l)- l_h(v)
\\
&= \underbrace{a_h( \pi_h u^e,v) - a((\pi_h u^e)^l , v^l)}_{I} 
+ \underbrace{j_h(\pi_h u^e,v)}_{II}
\\ \nonumber
&\qquad + \underbrace{a((\pi_h u^e)^l - u,v^l)}_{III} 
+ \underbrace{l(v^l)- l_h(v)}_{IV}
\\
&=I + II + III + IV
\end{align}
Here terms $I$ and $IV$ gives rise to the second and 
third terms on the right hand side in (\ref{eq:strang}) 
and $II$ and $III$ can be estimated as follows 
\begin{equation}\label{eq:proofstrangb}
|II| + |III| \lesssim h^{3/2} \| u \|_{H^2(\Gamma)} \tn v \tn_h 
\end{equation}
which together with (\ref{eq:infsup}) yields (\ref{eq:strang}). 
It remains to verify (\ref{eq:proofstrangb}).
\paragraph{Term $\bfI\bfI$.} This term is immediately estimated using (\ref{eq:interpolface}) as follows
\begin{align}
|II|&= |j_h(\pi_h u^e - u^e,v)| 
\\
&\lesssim h\tn \pi_h u^e - u^e \tn_{\mcF_h} \tn v \tn_{\mcF_h}
\\
&\lesssim h^{3/2}\| u \|_{H^2(\Gamma)} \tn v \tn_h
\end{align}

\paragraph{Term $\bfI \bfI \bfI$.} Using Green's formula, 
the Cauchy-Schwarz inequality, and the interpolation 
estimate (\ref{eq:interpolsurf}) we get
\begin{align}
a((\pi_h u^e)^l - u,v^l)
&=(\beta\cdot \nablas (\pi_h u^e)^l - u),v)_{\Gamma} 
+(\alpha (\pi_h u^e)^l - u),v)_{\Gamma} 
\\
&=-( (\pi_h u^e)^l - u,\beta \cdot \nablas v)_\Gamma 
+ ((\alpha - \divs \beta)(\pi_h u^e)^l - u),v)_\Gamma
\\
&\lesssim (1 + h^{-1})^{1/2} \| u - (\pi_h u^e)^l\|_\Gamma 
\Big(h \| \beta \cdot \nablas v \|^2_{\Gamma} + \| v \|^2_{\Gamma} \Big)^{1/2} 
\\
&\lesssim h^{-1/2} \| u - ( \pi_h u^e)^l\|_\Gamma \tn v \tn_h 
\\
&\lesssim h^{3/2} \| u \|_{H^2(\Gamma)} \tn v \tn_h
\end{align}
In the last step we used the estimate 
\begin{align}
\|\beta\cdot \nablas v \|^2_\Gamma 
&=\| \beta \cdot \nablas v \|^2_{\mcK_h^l}
\\
&=\|\,|B|\beta \cdot \nablash v \|^2_{\mcK_h}
\\
&\lesssim \|(|B|B^{-1}\beta - \beta_h) \cdot \nablash v \|^2_{\mcK_h}
+ \|\beta_h \cdot \nablash v \|^2_{\mcK_h}
\\
&\lesssim h^{-1}\||B|B^{-1}\beta - \beta_h\|^2_{L^\infty(\mcK_h)} 
\|\nabla v \|^2_{\mcT_h}
+ \|\beta_h \cdot \nablash v \|^2_{\mcK_h}
\\
&\lesssim h^{-3}\||B|B^{-1}\beta - \beta_h\|^2_{L^\infty(\mcK_h)}
\| v \|^2_{\mcT_h}
+ \|\beta_h \cdot \nablash v \|^2_{\mcK_h}
\\
&\lesssim h^{-2}\||B|B^{-1}\beta - \beta_h\|^2_{L^\infty(\mcK_h)}
\Big( \| v \|^2_{\mcK_h} + \tn v \tn^2_{\mcF_h} \Big)
+ \|\beta_h \cdot \nablash v \|^2_{\mcK_h}
\\
&\lesssim \tn v \tn^2_h
\end{align}
where we changed domain of integration from $\mcK_h^l$ to 
$\mcK_h$, added and subtracted $\beta_h$, used the inverse 
estimate (\ref{eq:inverseKtoT}) to pass from $\mcK_h$ to 
$\mcT_h$ in the second factor of the first term, employed 
Lemma \ref{lem:TechnicalA}, and finally the assumption that 
$\||B|B^{-1}\beta - \beta_h\|^2_{L^\infty(\mcK_h)}\lesssim h$, 
which follows from the stronger assumption (\ref{eq:bethedgeestassumption}).
\end{proof}

\subsection{Quadrature Errors}
\label{Sec:QuadErrors}
\paragraph{Definition of $\bfbeta_h$.} Changing domain of integration 
to $\Gammah$ and using the identity (\ref{eq:tanderlift}) 
for the tangential derivative of a lifted function we obtain
\begin{align}
(\beta\cdot\nablas v^l, w^l)_{\mcK_h^l} 
-(\beta_h\cdot\nablash v, w)_{\mcK_h}
&= (|B|(\beta\cdot\nablas v^l)^e - \beta_h \cdot(\nablash v), 
w)_{\mcK_h} 
\\
&= (|B|(\beta\cdot B^{-T}\nablash v)^e - \beta_h\cdot \nablash v, w)_{\mcK_h}
\\ \label{eq:quada}
&= ((|B|B^{-1}\beta - \beta_h) \cdot \nablash v, w)_{\mcK_h}
\end{align}
We note that $\beta \mapsto |B|B^{-1}\beta$ is in fact a Piola 
mapping which maps tangent vectors on $\Gamma$ onto tangent 
vectors on $\Gamma_h$. 

Since $(\nablash v) w$ is a linear function on each $K\in\mcK_h$ taking 
\begin{equation}\label{def:betah}
\beta_h = P_{1,K} (|B|B^{-1}\beta)
\end{equation}  
where $P_{1,K}$ is the $L^2$ projection onto $P_1(K)$, the space of linear polynomials on $K$, the quadrature error is actually zero.

In practice, we can not compute the exact $L^2$ projection instead we must use a quadrature rule. Starting from (\ref{eq:quada}) we 
get the estimate
\begin{align}
\nonumber
&|(\beta\cdot\nablas (\pi_h u^e)^l, w^l)_{\mcK_h^l} 
-(\beta_h\cdot\nablash (\pi_h u^e), w)_{\mcK_h}|
\\
&\qquad \lesssim \||B|B^{-1}\beta - \beta_h\|_{L^\infty(\mcK_h)}  
\|\nablash (\pi_h u^e) \|_{\mcK_h} \|w\|_{\mcK_h}
\\
&\qquad \lesssim \||B|B^{-1}\beta - \beta_h\|_{L^\infty(\mcK_h)}  
\| u \|_{H^1(\Gamma)} \|w\|_{\mcK_h}
\end{align}
Here we used $H^1$ stability of the interpolant and the 
following estimate
\begin{equation}
\|\nablash (\pi_h u^e) \|^2_{\mcK_h}
\lesssim 
h^{-1}\|\nabla (\pi_h u^e) \|^2_{\mcT_h}
\lesssim 
h^{-1}\|\nabla u^e \|^2_{\mcT_h}
\lesssim
h^{-1}\|\nablas u\|^2_{U_{\delta}(\Gamma)}
\lesssim 
\|\nablas u\|^2_{\Gamma}
\end{equation}
where $\mcT_h \subset U_{\delta}(\Gamma)$ with $\delta\sim h$.

\begin{rem} Note that the fact that $\pi_h u^e$ appears in the first slot in the bilinear form is crucial since we get $L^2$ control over the the full tangent gradient 
$\nablas (\pi_h u^e)^l$ using stability of the interpolant, 
while the corresponding control of the discrete solution 
$u_h$ provided by Proposition \ref{prop:stabB} is only 
$h^{3/4} \| \nablas u_h\|_{\mcK_h} \lesssim \tn u_h \tn_h 
\lesssim \sup_{v\in V_h} \frac{A_h(u_h,v)}{\tn v \tn_h} 
\lesssim \|f_h \|_{\mcK_h}$ indicating a higher demand on 
the accuracy of $\beta_h$ in order to achieve optimal order 
of convergence.
\end{rem}

%where we used the control 
%over $h^{3/4}\|\nablash v\|_{\mcK_h}$, see Proposition \ref{prop:stabB}. Since we require an error less or equal 
%to $h^{3/2}$ we need at least $h^{3/2 + 3/4}  = h^{9/4}$ 
%accuracy in $\beta_h$. Thus one alternative is to construct 
%a quadratic polynomial approximation $\beta_h$ of 
%$|B|B^{-1}\beta$ for each element $K\in \mcK_h$ satisfying
%\begin{equation}\label{eq:betahbound}
%\| (|B| B^{-1}\beta -\beta_h) \|_{L^\infty(\mcK_h)} \lesssim h^3
%\end{equation}
%and then use a quadrature rule that is exact for cubic polynomials. To construct such a polynomial we may construct a shape regular triangle $T_K$ of diameter $h$ containing $K$ 
%and use the quadratic Lagrange interpolant on $T_K$ and then 
%take the restriction to $K$. This procedure leads to a quadrature error that can be bounded as follows
%\begin{multline}
%|((|B| B^{-1}\beta -\beta_h)\cdot \nablash v,w)_{\mcK_h}|
%\lesssim 
%\| (|B| B^{-1}\beta -\beta_h) \|_{L^\infty(\mcK_h)} \|\nablash v \|_{\mcK_h}\|w \|_{\mcK_h} 
%\\
%\lesssim
%h^3\|\nablash v \|_{\mcK_h} \|w \|_{\mcK_h} \lesssim h^2 \tn v \tn_h \tn w \tn_h
%\end{multline}
%\begin{rem} Note that if we use a linear approximation, which is 
%easier to implement, we would 
%get second order accuracy in $\beta_h$ and we would instead 
%get a bound of order $h^{2-3/4} = h^{5/4}$, indicating a loss 
%of order $h^{1/4}$ compared to the optimal order of convergence $h^{3/2}$.
%\end{rem}
\paragraph{Definition of $\boldsymbol{\alpha}_{\boldsymbol{h}}$ and $\boldsymbol{f}_{\boldsymbol{h}}$.}
Similarly for the approximation of $\alpha$ and $f$ we let $\alpha_h$ and $f_h$ be linear approximations of $\alpha$ and $f$ on each element $K$. For  
$(\alpha_h u,v)_K$ a quadrature rule that is exact for cubic polynomials is used while for $(f_h,v)_K$ a quadratic rule is enough. We have the estimates
\begin{multline}
|(\alpha (\pi_h u^e)^l,v^l)_{\mcK_h^l} - (\alpha_h (\pi_h u^e),v)_{\mcK_h^l}|
=|((|B|\alpha^e -\alpha_h) (\pi_h u^e),v)_{\mcK_h}|
\\
\lesssim h^2 \| \pi_h u^e\|_{\mcK_h} \| v \|_{\mcK_h}
\lesssim h^2 \| u \|_{\Gamma} \tn v \tn_h
\end{multline}
and
\begin{equation}
|(f,v^l)_{\mcK_h^l} 
- (f_h,v^l)_{\mcK_h}| 
= |(|B|f^e - f_h,v)_{\mcK_h}|
\lesssim h^2 \| v \|_{\mcK_h}
\lesssim h^2 \tn v \tn_h
\end{equation}
 
We summarize our results in the following Lemma. 
 
\begin{lem}\label{lem:quadratureestimates} Let $\beta_h$ be an elementwise quadratic polynomial approximation of $|B|B^{-1}\beta$,  
$\alpha_h$ and $f_h$ elementwise linear approximations of $\alpha$ and $f$. Assuming that 
\begin{equation}\label{eq:assumptioncoefficients}
\||B| B^{-1} \beta -\beta_h \|_{L^\infty(\mcK_h)} \lesssim h^2, 
\quad 
\| |B| \alpha - \alpha_h \|_{^\infty(\mcK_h)} \lesssim h^2,
\quad 
\| |B| f - f_h \|_{L^\infty(\mcK_h)} \lesssim h^2
\end{equation}
then 
\begin{align}
|a((\pi_h u^e)^l,v^l) - a_h((\pi_h u^e),v)| 
&\lesssim h^{2}\|u\|_{H^1(\Gamma)} \tn w \tn_h
\quad \forall v \in V_h
\\
|l(v^l) - l_h(v)| &\lesssim h^2 \| v \|_\Gammah 
\quad \forall v \in V_h
\end{align}
\end{lem}

\begin{rem} Note that we do not have to take $|B|$ 
into account in (\ref{eq:assumptioncoefficients}) since 
$|\,|B|-1\,| \sim h^2$. Thus we could instead use the 
simplified assumptions 
\begin{equation}\label{eq:assumptioncoefficientssimple}
\|B^{-1} \beta -\beta_h \|_{L^\infty(\mcK_h)} \lesssim h^2, 
\quad 
\| \alpha - \alpha_h \|_{^\infty(\mcK_h)} \lesssim h^2,
\quad 
\| f - f_h \|_{L^\infty(\mcK_h)} \lesssim h^2
\end{equation}
Furthermore, $B^{-1} = P_\Gammah + O(h^2)$, and thus we 
may simplify the assumption on $\beta_h$ even further 
and assume that 
\begin{equation}
\|P_\Gammah \beta^e -\beta_h \|_{L^\infty(\mcK_h)} \lesssim h^2
\end{equation}
\end{rem}

We finally verify that $\beta_h$ satisfies the assumption (\ref{eq:bethedgeestassumption}).

\begin{lem}\label{lem:betahedgeest} Let $\beta_h$ be 
as in Lemma \ref{lem:quadratureestimates}. Then it 
holds
%
%elementwise quadratic polynomial approximation of 
%$|B|B^{-1}\beta$ such that
%\begin{equation}
%\||B| B^{-1} \beta -\beta_h \|_{L^\infty(\mcK_h)} \lesssim h^3
%\end{equation}
%then 
\begin{equation}
\|[n_E\cdot\beta_h]\|_{L^\infty(\mcE_h)} \lesssim h^2
\end{equation}
\end{lem}
\begin{proof} Splitting the error by adding and subtracting 
$\beta$ we get
\begin{align}
\|[n_E \cdot \beta_h]\|_{L^\infty(E)} 
&\leq 
\|(n_E \cdot (\beta_h-\beta))^+ \|_{L^\infty(E)}
\\ \nonumber
&\qquad + 
 \|(n_E \cdot (\beta_h-\beta))^- \|_{L^\infty(E)}
+ 
\|[n_E]\cdot \beta\|_{L^\infty(E)}
\\
&=I + II + III
\end{align}
Here the last term is $O(h^2)$ and it remains to estimate 
the first two terms which are of the same form. 
\paragraph{Terms $\bfI$ and $\bfI\bfI$.}
Using the definition of $\beta_h$ we have 
\begin{align}
|n_E \cdot (\beta_h-\beta) \|_{L^\infty(E)}
&\leq
\|n_E \cdot (\beta_h - |B|B^{-1}\beta) \|_{L^\infty(E)}
+ \|n_E \cdot (|B|B^{-1}\beta - \beta) \|_{L^\infty(E)}
\\
&\lesssim h^3 + h^2
\end{align}
Here the second term was estimated as follows 
\begin{align}
&\|n_E \cdot (|B|B^{-1}\beta - \beta) \|_{L^\infty(E)}
\nonumber
\\
&\qquad = \|n_E \cdot (|B|B^{-1}\beta - P_{\Gammah} \beta) \|_{L^\infty(E)}
\\
&\qquad \lesssim 
\|n_E \cdot B^{-1} ((|B| - 1) \beta\|_{L^\infty(E)} 
+ \|n_E \cdot B^{-1}(\Ps - B P_{\Gammah})\beta  \|_{L^\infty(E)}
\\
&\qquad \lesssim h^2
\end{align}
where we used the bounds (\ref{BBTbound}) for $B$. 

\paragraph{Term $\bfI\bfI\bfI$.}
Since $\beta$ is a smooth tangent vector field on 
$\Gamma$ we have 
\begin{align}
[n_E \cdot \beta ] &= [n_E] \cdot \beta = [n_E] \cdot t^e_E |t_E^e \cdot\beta^e| 
\end{align}
where $t_E \in T_{p(x)}(\Gamma)$ is the unit tangent vector 
at $p(x)$ to the exact surface that is orthogonal to edge $E$. Thus it remains to estimate $[n_E] \cdot t^e_E$. We have the 
identity 
\begin{equation}
[n_E] \cdot t^e_E 
= 
([n_K] \times e_E) \cdot (n \times e_E)
=[n_K]\cdot n \sim h^2
\end{equation}
\end{proof}

\subsection{Error Estimates}
\begin{thm} Let $u$ be the solution to  \eqref{eq:conva} and $u_h$ 
the finite element approximation defined by \eqref{eq:fem}, then the following estimate holds 
\begin{equation}
\tn u^e - u_h \tn_h \lesssim h^{3/2}
\end{equation}
for all $h\in (0,h_0]$ with $h_0$ small enough.
\end{thm}
\begin{proof} Adding and subtracting an interpolant
\begin{equation}
\tn u^e - u_h \tn_h \leq \tn u^e - \pi_h u^e \tn_h + \tn \pi_h u^e - u_h \tn_h
\end{equation}
Here the first term is estimated using the interpolation error 
estimate (\ref{eq:interpolenergy}) and for the second we apply the 
Strang Lemma \ref{lem:strang} together with the quadrature error estimates in Lemma \ref{lem:quadratureestimates}.
\end{proof}

\begin{thm}\label{thm:errorest} Let $u$ be the solution to  \eqref{eq:conva} and $u_h$ the finite element approximation 
defined by \eqref{eq:fem}, then the following estimate holds 
\begin{equation}
\| \nablash(u^e - u_h) \|_{\mcK_h} \lesssim h^{3/4} 
\end{equation}
for all $h\in (0,h_0]$ with $h_0$ small enough.
\end{thm}

\begin{proof} We have the estimates
\begin{align}
\| \nablash (u^e - u_h) \|_{\mcK_h} 
&\lesssim \| \nablash ( u^e - \pi_h u^e)  \|_{\mcK_h}
+ \| \nablash (\pi_h u^e - u_h) \|_{\mcK_h}
\\
&\lesssim \| \nablash ( u^e - \pi_h u^e)  \|_{\mcK_h} 
+ h^{-3/4} \tn \pi_h u^e - u_h \tn_h
\\ 
&\lesssim \| \nablash ( u^e - \pi_h u^e)  \|_{\mcK_h}
+ h^{-3/4} \tn \pi_h u^e - u^e \tn_h 
\\ \nonumber
&\qquad
+ h^{-3/4} \tn u^e - u_h \tn_h
\\
&\lesssim h + h^{-3/4} h^{3/2} + h^{-3/4} h^{3/2} 
\\
&\lesssim h^{3/4}
\end{align}
Here we added and subtracted an interpolant and used the 
triangle inequality, used the stability estimate in Proposition 
\ref{prop:stabB}, added and subtracted an interpolant in 
the second term and used the triangle inequality, 
used interpolation error estimates (\ref{eq:interpolsurf}) 
and  (\ref{eq:interpolenergy}) to estimate the first and the 
second term and finally Theorem \ref{thm:errorest} to estimate 
the third term, which conclude the proof of the desired estimate.
\end{proof}

\begin{rem} We note that these two error estimates are completely 
analogous to the estimates obtained in \cite{BuHa04} for the corresponding method on standard triangular meshes in the plane. 
\end{rem}

\section{Condition Number Estimate}
Let $\{\varphi_i\}_{i=1}^N$ be the standard piecewise linear 
basis functions associated with the nodes in $\mcT_h$ and 
let $\mcA$ be the stiffness matrix with elements 
$a_{ij} = A_h(\varphi_i,\varphi_j)$. 
We recall that the condition number is defined by
\begin{equation}\label{cond_def}
\kappa_h(\mcA) := | \mcA |_{\IR^N} |\mcA^{-1} |_{\IR^N}
\end{equation}
Using the ideas introduced in \cite{BuHaLa14}, we may prove 
the following bound on the condition number of the matrix.

\begin{thm} The condition number of the stiffness matrix $\mcA$
satisfies the estimate
\begin{equation}
\kappa_h(\mcA)\lesssim h^{-2}
\end{equation}
for all $h \in (0,h_0]$ with $h_0$ small enough.
\end{thm} 
\begin{proof}
First we note that if $v = \sum_{i=1}^N V_{i} \varphi_i$ 
and $\{\varphi_i\}_{i=1}^N$ is the usual nodal basis on 
$\mcT_h$ then the following well known estimates hold
\begin{equation}\label{rneqv}
c h^{-d/2} \| v \|_{\mcT_h} \leq | V |_{\IR^N} \leq C h^{-d/2}\| v \|_{\mcT_h}
\end{equation}
It follows from the definition (\ref{cond_def}) of the condition 
number that we need to estimate $| \mcA |_{\IR^N}$ and 
$|\mcA^{-1}|_{\IR^N}$. 

\paragraph{Estimate of $| \mcA |_{\IR^N}$.} We have
\begin{align}
|\mcA V|_{\IR^N} &= \sup_{W \in \IR^N \setminus 0} \frac{(W,\mcA V)_{\IR^N}}{| W |_{\IR^N}}
\\
&= \sup_{w \in V_h \setminus 0 }  \frac{A_h(v,w)}{| W |_{\IR^N}}
\\
&\lesssim h^{d- 2}| V|_{\IR^N}
\end{align}
Here we used the following continuity of $A_h(\cdot,\cdot)$
\begin{align}
A_h(v,w) &\lesssim 
\| \beta_h \cdot \nabla v \|_{\mcK_h} \| w \|_{\mcK_h} 
+ h \tn v \tn_{\mcF_h} \tn w \tn_{\mcF_h}
\\
&\lesssim 
\Big( h \| \beta_h \cdot \nabla v \|_{\mcK_h}^2 
+ h \tn v \tn_{\mcF_h}^2\Big)^{1/2}
\Big(h^{-1}\| w \|^2_{\mcK_h} 
+ h \tn w \tn_{\mcF_h}^2\Big)^{1/2}
\\
&\lesssim h^{d-2} |V|_{\IR^N} | W |_{\IR^N}
\end{align}
In the last step we used the estimates
\begin{align}
h \| \beta_h \cdot \nabla v \|_{\mcK_h}^2 
+ h \tn v \tn_{\mcF_h}^2 
&\lesssim 
\| \beta_h \cdot \nabla v \|_{\mcT_h}^2 
+ \| \nabla v \|_{\mcT_h}^2
\lesssim 
h^{-2} \| v \|_{\mcT_h}^2 
\lesssim 
h^{d-2}|V|_{\IR^N}^2
\end{align}
where we used the inverse estimates (\ref{eq:inverseKtoT}) 
and (\ref{eq:inverseFtoT}) to pass from $\mcK_h$ and $\mcF_h$ 
to $\mcT_h$, an inverse estimate to remove the gradient, 
and finally the equivalence (\ref{rneqv}); and 
\begin{equation}
h^{-1}\| w \|^2_{\mcK_h} 
+ h \tn w \tn_{\mcF_h}^2
\lesssim 
h^{-2} \|w\|^2_{\mcT_h} + \| \nabla w \|_{\mcT_h}^2
\lesssim 
h^{-2} \|w\|^2_{\mcT_h}
\lesssim 
h^{d-2} |W|^2_{\IR^N}
\end{equation}
where we used the same sequence of estimates.
It follows that  
\begin{equation}\label{Abound}
| \mcA |_{\IR^N} \lesssim h^{d-2}
\end{equation}

\paragraph{Estimate of $|\mcA^{-1} |_{\IR^N}$.} We note that 
using \eqref{rneqv} and Lemma \ref{lem:TechnicalA} we have 
\begin{equation}\label{eq:condtechnicala}
 h^d | V |^2_{\IR^N}  
 \lesssim \|v\|^2_{\mcT_h} 
 \lesssim h \Big( \|v\|_{\mcK_h}^2  + \tn v \tn_{\mcF_h}^2 \Big) 
 \lesssim  \tn v \tn_h^2
\end{equation}
where in the last step we used the fact that $h \in (0,h_0]$ and 
the definition (\ref{eq:normtrippleh}) of $\tn \cdot \tn_h$. Starting 
from \eqref{eq:condtechnicala} and using the inf-sup condition 
(\ref{eq:stabest})  we obtain 
\begin{align}
 | V |_{\IR^N} 
 \lesssim h^{-d/2} \tn
 v \tn_h 
&\lesssim h^{-d/2} \sup_{w \in V_h \setminus \{0\} }  \frac{A_h(v,w)}{\tn w
  \tn_h} 
\\  
&\qquad  
\lesssim \sup_{W \in \IR^N\setminus \{0\}}
 h^{-d/2} \frac{ | \mcA V|_{\IR^N}  |W|_{\IR^N}
}{h^{{d}/{2}} |W|_{\IR^N}} \lesssim h^{-d}  | \mcA V|_{\IR^N} 
\end{align}
Here we used (\ref{eq:condtechnicala}), 
$h^{d/2} |W|_{\IR^N} \lesssim \tn w \tn_h$,
to replace $\tn w \tn_h$ by $h^{d/2} |W|_{\IR^N} $  in the denominator.
Setting $V= \mcA^{-1} X$, $X \in \IR^N$, we obtain
\begin{equation}\label{Ainvbound}
| \mcA^{-1}  |_{\IR^N} \lesssim h^{-d}  
\end{equation}

\paragraph{Conclusion.} The claim follows by using the bounds \eqref{Abound} and \eqref{Ainvbound} in the definition \eqref{cond_def}.
\end{proof}
\section{Numerical Examples}
\subsection{Convergence Study}
We consider an example where the surface $\Gamma$ is a ring torus and given by the zero level set of the the signed distance function $\rho=\sqrt{z^2+\lp \sqrt{x^2+y^2}-R \rp^2}-r$, with $R=1$ and $r=1/2$,  we choose $\alpha=1$, 
\begin{equation}
\beta=P_\Gamma (x^2yz, x, yz^3)
\end{equation}
and $f$ such that the exact solution is  
\begin{align}\label{eq:exactsol}
u&=(0.5x+(x-1)^2+0.5y+(y-1))e^{(-x(x-1)-y(y-1))}
\end{align}
\begin{figure}[h]
\begin{center} 
\includegraphics[width=1\textwidth]{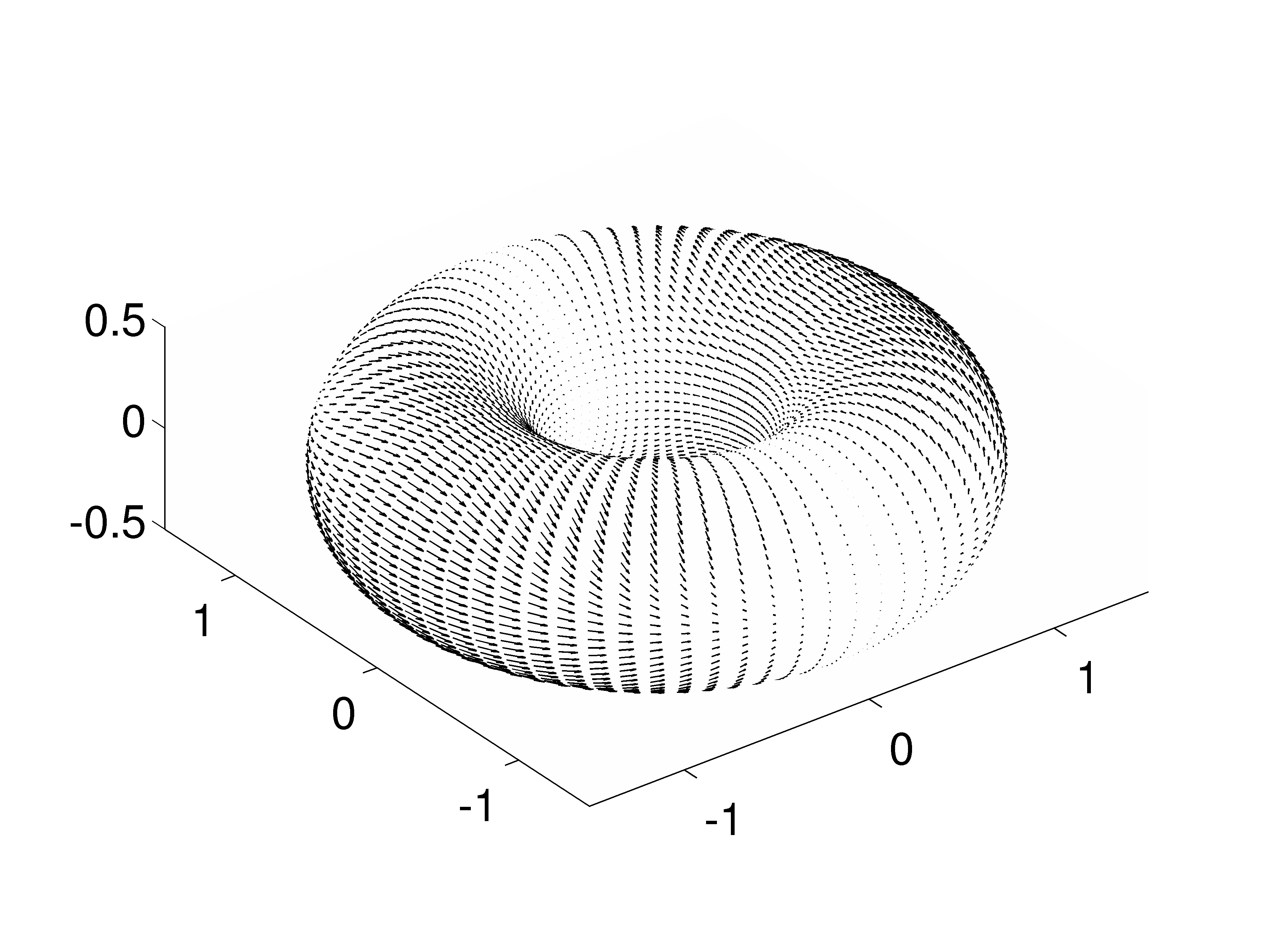}
\caption{The vector field $\beta$ on the torus.  \label{fig:beta}}
\end{center}
\end{figure}

The vector field $\beta$ is shown in Fig.~\ref{fig:beta}. A structured mesh $\mcT_{0,h}$ consisting of tetrahedra on the domain $[-1.6, 1.6] \times [-1.6, 1.6] \times [-0.6, 0.6]$ is generated independently of the position of the torus. The mesh size parameter is defined as $h=h_x=h_y=h_z$. An approximate distance function $\rho_h$ is constructed using the nodal interpolant $\pi_h \rho$ on the background mesh and $\Gamma_h$ is the zero levelset of $\rho_h$ and $n_h$ is the piecewise constant unit normal to $\Gamma_h$. The triangulation of $\Gamma_h$ is shown in Fig.~\ref{fig:meshtorus}. We use the proposed method with the stabilization parameter chosen as $c_F=10^{-2}$. A direct solver was used to solve the linear systems. The solution $u_h$ for $h=0.2$ is shown in Fig.~\ref{fig:sol}. We compare our approximation $u_h$ with the exact solution $u$ given in equation~\eqref{eq:exactsol}. The convergence of $u_h$ in  the $L^2$ norm and the energy norm are shown in Fig.~\ref{fig:conv}. We observe second order convergence in the $L^2$ norm and as expected a convergence order of $1.5$ in the energy norm. The error in the gradient $\| \nablash(u^e - u_h) \|_{\mcK_h}$ versus mesh size is shown in Fig.~\ref{fig:convgrad} and the observed convergence order is slightly better than $3/4$ for the finest meshes. 
%The condition number of the matrix is shown for different mesh sizes in Fig.~\ref{fig:cond}.
\begin{figure}
\begin{center} 
\includegraphics[width=0.6\textwidth]{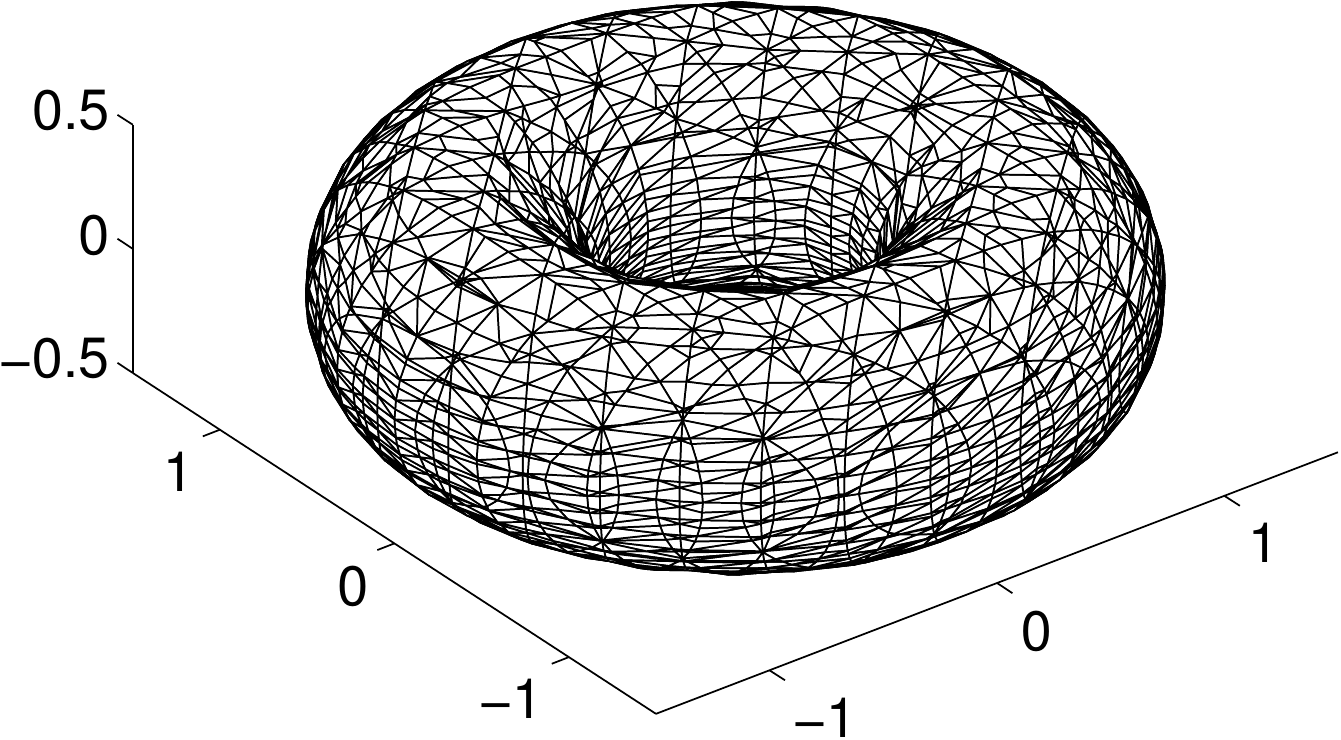}
\caption{The triangulation of $\Gamma_h$ for $h=0.2$.  \label{fig:meshtorus}}
\end{center}
\end{figure}

\begin{figure}
\begin{center} 
\includegraphics[width=0.7\textwidth]{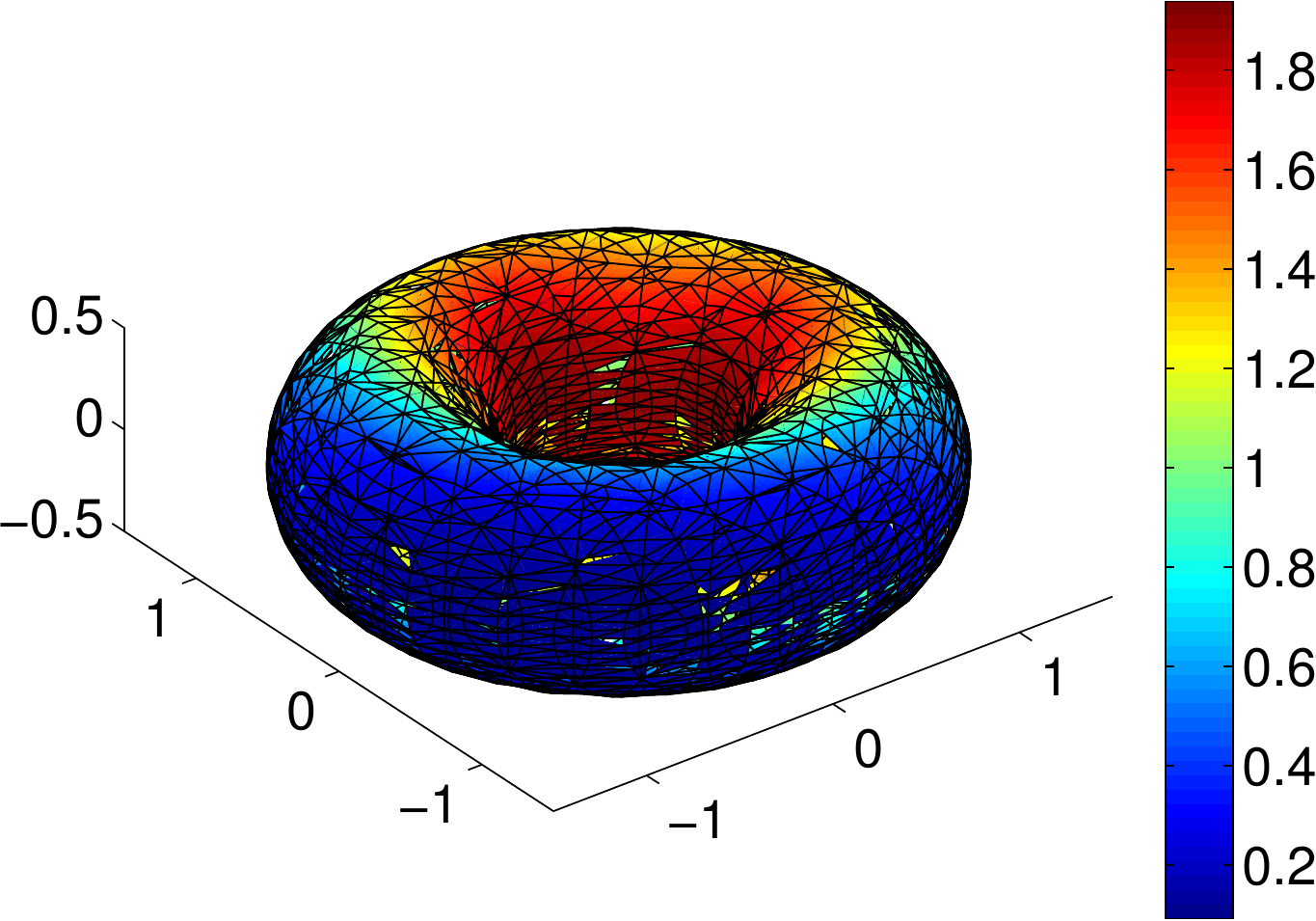}
\caption{The solution $u_h$ for $h=0.2$. \label{fig:sol}}
\end{center}
\end{figure}

\begin{figure}
\begin{center} 
\includegraphics[width=0.7\textwidth]{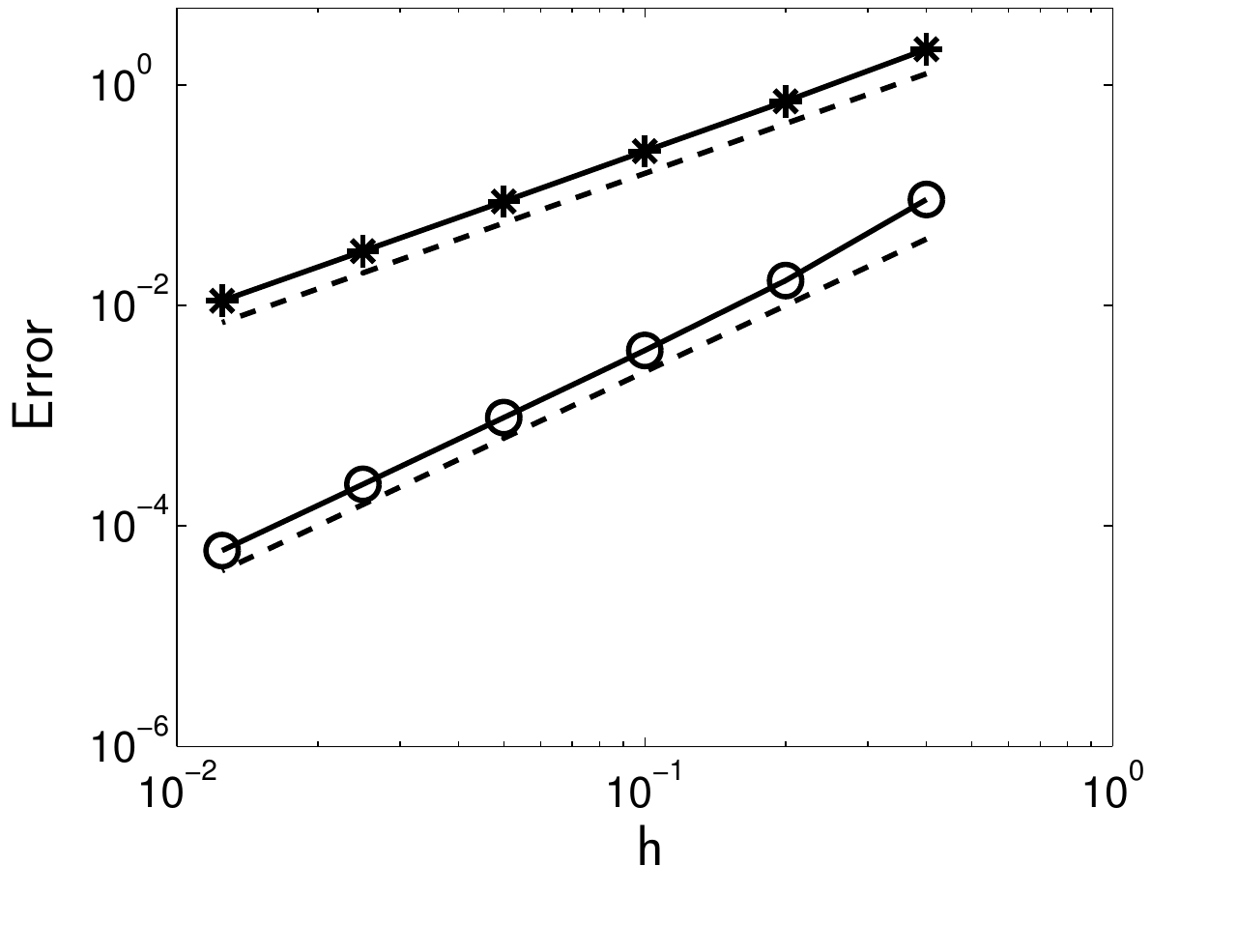}
\caption{Convergence of $u_h$.  Circles represent the error measured in the $L^2$ norm ($\|u^e - u_h \|_{\mcK_h} $) and stars represent the error in the energy norm ($\tn u^e - u_h \tn_h$). The dashed lines are proportional to  $h^{2}$ and $h^{3/2}$. 
\label{fig:conv}}
\end{center}
\end{figure}

\begin{figure}
\begin{center} 
\includegraphics[width=0.7\textwidth]{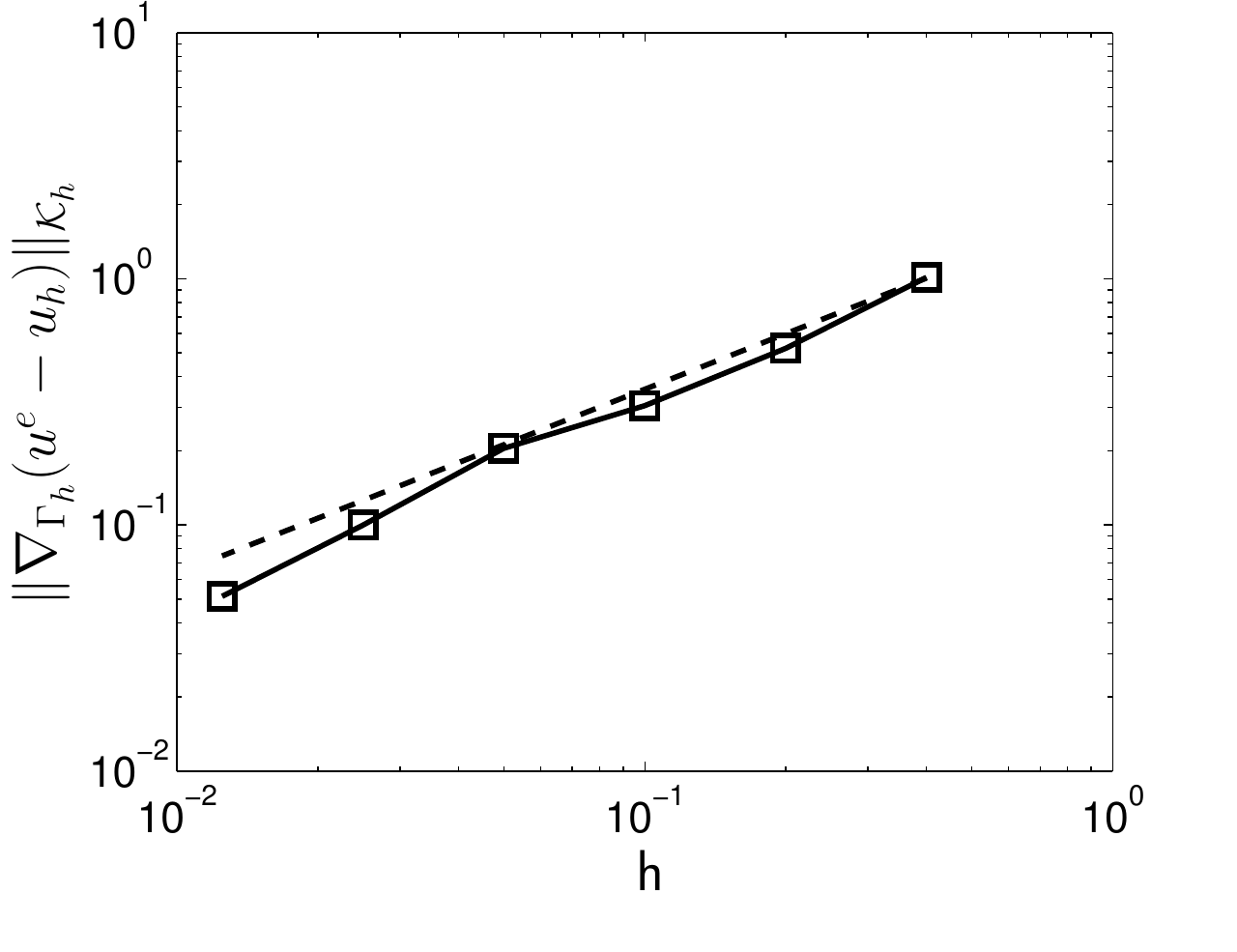}
\caption{Convergence of $\| \nablash(u^e - u_h) \|_{\mcK_h}$. The dashed line is proportional to $h^{3/4}$.
\label{fig:convgrad}}
\end{center}
\end{figure}

\subsection{Condition Number Study}

We compute the condition number of the stiffness matrix for a sequence of uniformly refined meshes and different values of the stabilization parameter $c_F$. We find that the asymptotic behavior as the meshsize 
tend to zero is $O(h^2)$, while on coarser meshes for larger values of 
$c_F$ the behavior is closer to $O(h)$.  
%The behavior for large $c_F$ and coarse meshes may be understood by observing that the condition number of the matrix associated with the stabilization term $j_h$ is $O(h)$.

\begin{figure}
\begin{center} 
\includegraphics[width=0.6\textwidth]{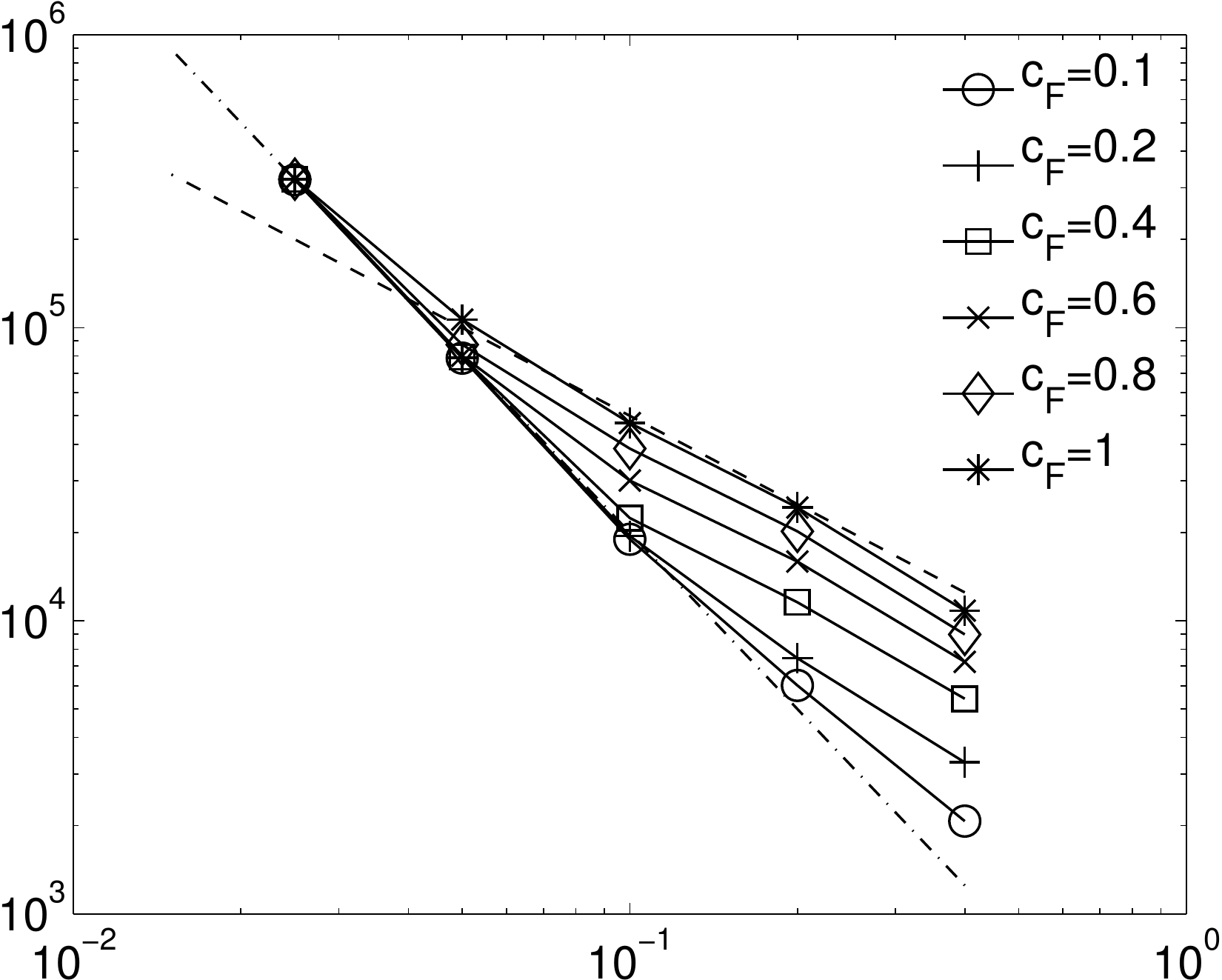}
\caption{The condition number of the matrix versus mesh size. The dashed lines are proportional to $h^{-1}$ and $h^{-2}$. \label{fig:cond}}
\end{center}
\end{figure}

%\section{To Do}
%
%\begin{itemize}
%\item Existence theory for first order PDEs on manifolds. With 
%and without boundary. We will try to skip boundary in this paper.
%\item Add something on inclusion of boundary with weak inflow conditions. We will try to skip boundary in this paper.
%\item Numerical example.
%\end{itemize}

\bibliographystyle{plain}
  \bibliography{ref}
\end{document}